\documentclass[11pt]{amsart}

 \usepackage{amsfonts,graphics,amsmath,amsthm,amsfonts,amscd,amssymb,amsmath,latexsym,multicol,
 mathrsfs}
\usepackage{epsfig,url}
\usepackage{flafter}
\usepackage{fancyhdr}
\usepackage{hyperref}
\hypersetup{colorlinks=true, linkcolor=black}

 \usepackage[matrix, arrow]{xy}

\DeclareMathOperator{\lct}{lct}

\DeclareMathOperator{\Pic}{Pic}
\DeclareMathOperator{\Proj}{Proj}

\DeclareMathOperator{\Supp}{Supp}

\DeclareMathOperator{\vol}{vol}
\DeclareMathOperator{\Bir}{Bir}

 \numberwithin{equation}{subsection}
 \numberwithin{footnote}{subsection}

 \newtheorem{cor}[subsection]{Corollary}
 \newtheorem{lem}[subsection]{Lemma}
 \newtheorem{prop}[subsection]{Proposition}
 \newtheorem{thm}[subsection]{Theorem}

{
\theoremstyle{upright}

 \newtheorem{exa}[subsection]{Example}

}

 \newcommand{\N}{\mathbb N}
 \newcommand{\PP}{\mathbb P}
 
 \newcommand{\Q}{\mathbb Q}
 \newcommand{\R}{\mathbb R}
 
  \newcommand{\C}{\mathbb C}
 \newcommand{\bir}{\dashrightarrow}
 \newcommand{\rddown}[1]{\left\lfloor{#1}\right\rfloor} 

\title{\large S\MakeLowercase{ingularities of linear systems and boundedness of }F\MakeLowercase{ano varieties}}
\author{\large C\MakeLowercase{aucher} B\MakeLowercase{irkar}}
\date{\today}
\begin{document}
\maketitle

\begin{abstract}
We study log canonical thresholds (also called global log canonical threshold or $\alpha$-invariant) of $\R$-linear 
systems. We prove existence of positive lower bounds in different settings, in particular, proving a conjecture of Ambro.    
We then show that the Borisov-Alexeev-Borisov conjecture holds, that is, 
given a natural number $d$ and a positive real number $\epsilon$, the set of Fano varieties of dimension $d$  
with $\epsilon$-log canonical singularities forms a bounded family. This implies that birational automorphism groups of rationally 
connected varieties are Jordan which in particular answers a question of Serre. Next we show that if the log canonical 
threshold of the anti-canonical system of a Fano variety is at most one, then it is 
computed by some divisor, answering a question of Tian in this case. 
\end{abstract}

\tableofcontents


\section{\bf Introduction}

We work over an algebraically closed field of characteristic zero unless stated otherwise.\\

{\textbf{\sffamily{Boundedness of singular Fano varieties.}}}
 A normal projective variety $X$ is \emph{Fano} if $-K_X$ is ample and if $X$ has log canonical singularities.  
Fano varieties are among the most extensively studied varieties because of their rich geometry.
They are of great importance from the point of view of 
birational geometry, differential geometry, arithmetic geometry, derived categories, mirror symmetry, etc.

Given a smooth projective variety $W$ with $K_W$ not pseudo-effective, the minimal model program 
produces a birational model $Y$ of $W$ together with a Mori fibre space structure $Y\to Z$ [\ref{BCHM}]. A  
general fibre of $Y\to Z$ is a Fano variety $X$ with terminal singularities. Thus it is no surprise 
that Fano varieties constitute a fundamental class in birational geometry. It is important to understand 
them individually but also collectively in families for various reasons such as construction of 
moduli spaces.

In dimension one, there is only one Fano variety up to isomorphism which is $\PP^1$. In dimension two, there are many, 
in fact, infinitely many families. To get a better picture one needs to impose a bound on the singularities. For example, 
it is a classical result that the smooth Fano surfaces form a bounded family. 
More generally, the Fano surfaces with $\epsilon$-log canonical ($\epsilon$-lc) 
singularities form a bounded family [\ref{Alexeev}], 
for any fixed $\epsilon>0$ (see \ref{ss-pairs} for definition of singularities). The smaller is $\epsilon$ the larger is the family.   

In any given dimension, there is a bounded family of smooth Fano varieties [\ref{KMM-smooth-fano}] 
(also see [\ref{Nadel}][\ref{Campana}]). 
This is proved using geometry of rational curves. Unfortunately, this method does not work 
when one allows singularities. On the other hand, toric Fano varieties 
of given dimension with $\epsilon$-lc singularities also form a bounded family [\ref{A-L-Borisov}], for fixed $\epsilon>0$.
In this case, the method of proof is based on combinatorics. 

The results mentioned above led Alexeev [\ref{Alexeev}] and the Borisov brothers [\ref{A-L-Borisov}]
to conjecture that, in any given dimension, Fano varieties with $\epsilon$-lc singularities form a bounded family,
for fixed $\epsilon>0$. A generalised form of this statement, which is known in the literature 
as the Borisov-Alexeev-Borisov or the BAB conjecture, is our first result. 

\begin{thm}\label{t-BAB}
Let $d$ be a natural number and $\epsilon$ be a positive real number. Then the projective 
varieties $X$ such that  
\begin{itemize}

\item $(X,B)$ is $\epsilon$-lc of dimension $d$ for some boundary $B$, and 

\item  $-(K_X+B)$ is nef and big,\\
\end{itemize}
form a bounded family. 
\end{thm}

The theorem would not hold if one takes $\epsilon=0$ (see below): it already fails in dimension two, 
and in dimension three it fails even if we replace bounded by birationally bounded [\ref{Lin}].

In addition to the results mentioned earlier, there are numerous other partial cases of the 
theorem in the literature (also see Nikulin [\ref{Nikulin}][\ref{Nikulin-2}][\ref{Nikulin-3}] for related results 
in dimension two). Indeed boundedness was known for: 
\begin{itemize}
\item Fano $3$-folds with terminal singularities and Picard number one [\ref{kawamata-term-3-folds}],

\item Fano $3$-folds with canonical singularities  [\ref{KMMT-can-3-folds}], 

\item spherical Fano varieties [\ref{Alexeev-Brion}], 

\item  Fano $3$-folds with klt singularities and fixed Cartier index of $K_X$ [\ref{A-Borisov}], 

\item Fano varieties of given dimension with klt singularities and fixed Cartier index of $K_X$ [\ref{HMX2}];

\item Fano varieties $X$ of given dimension equipped with a boundary $\Delta$ such that 
$K_X+\Delta\equiv 0$, $(X,\Delta)$ is $\epsilon$-lc for fixed $\epsilon>0$, and such that the coefficients of 
$\Delta$ belong to a DCC set [\ref{HMX2}], or more generally when the coefficients of 
$\Delta$ are bounded from below away from zero [\ref{B-compl}].\\
\end{itemize}

We give two examples of unbounded families of klt Fano varieties of dimension two. 
\begin{exa}
\emph{
For each natural number $n\ge 2$, let $X_n$ be the projective cone over the rational curve of deg $n$ in $\PP^n$. 
Let $f\colon W_n\to X_n$ be given by the blowup of the vertex, and let $E$ be the exceptional curve. 
Using adjunction it is easy to show that 
$$
K_{W_n}+\frac{n-2}{n}E=f^*K_{X_n}.
$$
Thus $X_n$ is a $\frac{2}{n}$-lc Fano variety. As $n\to \infty$, the singularities of $X_n$ get worse.
Now $\{X_n \mid n\in \mathbb{N}, n\ge 2\}$ is not a bounded family otherwise the Cartier index of $K_{X_n}$ 
would have been bounded which in turn would imply that the coefficient of $E$ in the 
above formula belongs to a fixed finite set which is clearly not the case.}
\end{exa}

\begin{exa}
\emph{
Consider the pair
$$
(\mathbb{P}^2,\Delta=S+T+R)
$$ 
where $S,T,R$ are the coordinate lines.  
Let $V_1\to \mathbb{P}^2$ be the blowup of $x_0:=S\cap T$.
Let $V_2\to V_1$ be the blowup of $x_1:=S^\sim \cap E_1$ where $E_1$ is the exceptional divisor of 
$V_1\to \mathbb{P}^2$ and $S^\sim$ is the birational transform of $S$.
Similarly define $V_n\to V_{n-1}$ where in each step we blowup the intersection point of 
$S^\sim$ with the newest exceptional curve. The $V_n$ are all toric varieties, in particular, $-K_{V_n}$ is big.
Run an MMP on $-K_{V_n}$ and let $X_n'$ be the resulting model. Then 
$X_n'$ is a klt toric weak Fano variety. Let $X_n'\to X_n$ be the contraction defined by $-K_{X_n'}$. 
Then $X_n$ is a klt toric Fano variety. Now 
$\{X_n \mid n\in \mathbb{N}\}$ is not a bounded family 
because $\{V_n \mid n\in \mathbb{N}\}$ is not a bounded family: indeed if the $X_n$ form a bounded family, 
then each $K_{X_n}$ has a klt $m$-complement for some $m$ independent of $n$; but then 
each $K_{V_n}$ would also have a klt $m$-complement which implies the $V_n$ form a bounded family 
by [\ref{HX}]; this is a contradiction because the Picard number of $V_n$ is clearly not bounded.}\\
\end{exa}

It is easy to see that Theorem \ref{t-BAB} is equivalent to the following statement.

\begin{cor}\label{cor-BAB}
Let $d$ be a natural number and $\epsilon$ be a positive real number. Then the projective 
varieties $X$ such that  
\begin{itemize}

\item $(X,\Delta)$ is $\epsilon$-lc of dimension $d$ for some boundary $\Delta$, 

\item $K_X+\Delta\sim_\R 0$ and $\Delta$ is big,\\
\end{itemize}
form a bounded family. 
\end{cor}

 The corollary was previously known when the coefficients of 
$\Delta$ are in a fixed DCC set [\ref{HX}] or when the coefficients are bounded from below 
away from zero [\ref{B-compl}].

The pairs $(X,\Delta)$ in the corollary are \emph{log Calabi-Yau pairs} with big boundary. 
Viewed in this context one is immediately led to the question: what other classes of log Calabi-Yau 
pairs form bounded families? In fact one can ask a host of hard fundamental questions regarding 
boundedness, singularities, complements and linear systems, 
in the more general context of  log Calabi-Yau fibrations. 

For example, 
a conjecture of M$^{\rm c}$Kernan and Prokhorov [\ref{MP}] predicts that the set of projective rationally connected 
varieties $X$ such that $(X,\Delta)$ is log Calabi-Yau of given dimension with $\epsilon$-lc 
singularities for fixed $\epsilon>0$ forms a bounded family. Without the rational connectedness assumption 
the conjecture does not hold as, for example, all smooth K3 surfaces do not form a bounded family. The conjecture is a more 
general form of \ref{cor-BAB} since the $X$ in the corollary are of  Fano type hence automatically 
rationally connected. Here by $X$ being of Fano type we mean there is a boundary $B$ such that $(X,B)$ is 
klt and $-(K_X+B)$ is nef and big.

For a systematic treatment of generalised log Calabi-Yau fibrations, see [\ref{B-lcy-fibs}].\\

{\textbf{\sffamily{Jordan property of Cremona groups.}}}
Prokhorov and Shramov [\ref{Prokhorov-Shramov}] studied the birational automorphism group of 
algebraic varieties using techniques of birational geometry. They investigated the question whether 
such groups are \emph{Jordan}. They in particular showed that the BAB conjecture, that is 
Theorem \ref{t-BAB}, implies the Jordan property for rationally connected varieties.

A group $C$ is Jordan of index $h$ if  
for any finite subgroup $G$ of $C$ there is a normal abelian subgroup $H$ of $G$ of index at most $h$. 
When we say $C$ is Jordan we mean that it is Jordan of index $h$ for some $h$.

\begin{cor}\label{cor-bir-aut}
Let $d$ be a natural number. Then there is a natural number $h$ depending only on $d$ 
satisfying the following. Let $k$ be a field of characteristic zero (not necessarily algebraically closed)  
and $X$ be a rationally connected variety of dimension $d$ over $k$. Then  
 the birational automorphism group $\Bir(X)$ is Jordan of index $h$.
\end{cor}

The corollary follows immediately from Theorem \ref{t-BAB} and [\ref{Prokhorov-Shramov}, Theorem 1.8]. 
If we take $X=\PP^d_k$ in the corollary, then we deduce that the Cremona group 
${\rm Cr}_d(k):=\Bir(\PP^d_k)$ is Jordan, answering a question of Serre [\ref{Serre}, 6.1] which was 
the main motivation for the work [\ref{Prokhorov-Shramov}]. 

Note that $\Bir(X)$ is not Jordan for arbitrary algebraic varieties $X$. Indeed 
when $X$ is the product of $\PP^1_k$ and an abelian variety, then $\Bir(X)$ is not Jordan [\ref{Zarhin}]. 
However, if $X$ is a non-uniruled variety, then $\Bir(X)$ is always Jordan [\ref{Prokhorov-Shramov-2}].\\

{\textbf{\sffamily{Lc thresholds of $\R$-linear systems.}}}
Let $(X,B)$ be an lc pair. The \emph{lc threshold} of an $\R$-Cartier $\R$-divisor $L\ge 0$ with respect to $(X,B)$ 
is defined as 
$$
\lct(X,B,L):=\sup\{t\in \R \mid (X,B+tL) ~~\mbox{is lc}\}.
$$ 
Now let $A$ be an $\R$-Cartier $\R$-divisor. The $\R$-linear system of $A$ is 
$$
|A|_\R=\{L\ge 0 \mid L\sim_\R A\}.
$$
We then define the \emph{lc threshold} of $|A|_\R$ with respect to $(X,B)$ (also called global 
lc threshold or $\alpha$-invariant)  as 
$$
\lct(X,B,|A|_\R):=\inf\{\lct(X,B,L) \mid L\in |A|_\R\}
$$
which coincides with 
$$
\sup\{t\in \R \mid (X,B+tL) ~~\mbox{is lc for every} ~~L\in |A|_\R\}.
$$
One can similarly define the lc threshold of $|A|$ and $|A|_\Q$ but we will not need them.

Due to connections 
with the notion of stability and existence of K\"ahler-Einstein metrics, lc thresholds of $\R$-linear systems  
have attracted a lot of attention particularly 
when $A$ is ample. An important special case is when $X$ is Fano and $A=-K_X$.\\

{\textbf{\sffamily{Lc thresholds of anti-log canonical systems of Fano pairs.}}}
We were led to lc thresholds of Fano varieties for a quite different reason. 
The paper [\ref{B-compl}] reduces Theorem \ref{t-BAB} to existence of a positive lower bound 
for lc thresholds of anti-canonical systems of certain Fano varieties which is guaranteed by our next result.

\begin{thm}\label{t-bnd-lct-global}
Let $d$ be a natural number and $\epsilon$ be a positive real number. Then there is a 
positive real number $t$ depending only on $d,\epsilon$ satisfying the following. 
Assume 
\begin{itemize}

\item $(X,B)$ is a projective $\epsilon$-lc pair of dimension $d$, and

\item $A:=-(K_X+B)$ is nef and big.\\
\end{itemize}
Then 
$$
\lct(X,B,|A|_\R)\ge t.
$$
\end{thm}

\vspace{0.3cm}
Although one may try to derive the theorem from \ref{t-BAB} but we actually do the opposite, that is, we will 
use the theorem to prove \ref{t-BAB} (see \ref{t-from-lct-to-bnd-var} below). The theorem was conjectured by Ambro [\ref{Ambro}] who proved it in the toric case.
Jiang [\ref{Jiang}][\ref{Jiang-2}] proved it in dimension two. 
It is worth mentioning that they both try to relate lc thresholds to
boundedness of Fano's  but our approach is entirely different.

The lc threshold of an $\R$-linear system $|A|_\R$ is defined as an infimum of usual lc thresholds. 
Tian [\ref{Tian}, Question 1] asked whether the infimum is a minimum when $A=-K_X$ and $X$ is Fano.  
The question was reformulated and generalised to log Fano's in [\ref{cheltsov-shramov}, 
Conjecture 1.12].  The next result gives a positive 
answer when the lc threshold is at most $1$.

\begin{thm}\label{t-global-lct-attained}
Let $(X,B)$ be a projective klt pair such that $A:=-(K_X+B)$ is nef and big. 
Assume that 
$$
\lct(X,B,|A|_\R)\le 1.
$$  
Then there is $0\le D\sim_\R A$ such that 
$$
\lct(X,B,|A|_\R)=\lct(X,B,D).
$$
Moreover, if $B$ is a $\Q$-boundary, then we can choose $D\sim_\Q A$, hence 
in particular, the lc threshold is a rational number in this case.
\end{thm}

The theorem is not used in the rest of the paper and its proof in the case 
$$
\lct(X,B,|A|_\R)<1
$$ 
relies on [\ref{HMX2}] but does not rely on the other results of this paper nor on the results of [\ref{B-compl}]. 
Ivan Cheltsov informed us that Shokurov has an 
unpublished proof of the theorem in dimension two.\\

{\textbf{\sffamily{Lc thresholds of $\R$-linear systems with bounded degree.}}}
Next we treat lc thresholds associated with divisors on varieties, in a general setting. 
To obtain any useful boundedness result, one needs to impose certain boundedness conditions on the invariants of the 
divisor and the variety. 

\begin{thm}\label{t-bnd-lct}
Let $d,r$ be natural numbers and $\epsilon$ be a positive real number. 
Then  there is a positive real number $t$ depending only on $d,r,\epsilon$ satisfying the following. 
Assume 
\begin{itemize}

\item  $(X,B)$ is a projective $\epsilon$-lc pair of dimension $d$, 

\item $A$ is a very ample divisor on $X$ with $A^d\le r$,

\item $A-B$ is pseudo-effective, and 

\item $M\ge 0$ is an $\R$-Cartier $\R$-divisor with $A-M$ pseudo-effective.\\
\end{itemize}
Then  
$$
\lct(X,B,|M|_\R)\ge \lct(X,B,|A|_\R)\ge t.
$$
\end{thm}

\vspace{0.3cm}
This is one of the main ingredients of the proof of Theorem \ref{t-bnd-lct-global} but it is also interesting on its own. 
We explain briefly some of the assumptions of the theorem. The condition $A^d\le r$ means that 
$X$ belongs to a bounded family of varieties, actually, if we choose $A$ general in its linear system, 
then $(X,A)$ belongs to a bounded family of pairs. We can use the divisor $A$ to measure how ``large" 
other divisors are on $X$. Indeed, 
the pseudo-effectivity of $A-B$ and $A-M$, roughly speaking, say 
that the ``degree" of  $B$ and $M$ are bounded from above, that is,
$$
\deg_AB:=A^{d-1}B\le A^d\le r ~~~~~\mbox{and}~~~~~ \deg_AM:=A^{d-1}M\le A^d\le r.
$$ 
Without such boundedness assumptions, one would not find a positive lower bound for the lc threshold.
 For example, if $X=\PP^d$, then one can 
easily find $M$ with arbitrarily small lc threshold if the degree of $M$ is allowed to be large enough. 
The bound on the degree of $B$ is much more subtle.\\

{\textbf{\sffamily{Complements near a non-klt centre.}}}
We prove boundedness of certain ``complements" near a non-klt centre on a projective 
pair. For the definition of complements in the global setting, see \ref{ss-complements}.

\begin{thm}\label{t-compl-near-lcc}
Let $d,p$ be natural numbers. Then there exists a natural number 
$n$ depending only on $d,p$ satisfying the following. Assume 

\begin{itemize}
\item $(X,B)$ is a projective lc pair of dimension $d$,

\item $pB$ is integral, 

\item $M$ is a semi-ample Cartier divisor on $X$ defining a contraction $f\colon X\to Z$,

\item $X$ is of Fano type over $Z$, 

\item $M-(K_X+B)$ is nef and big, and 

\item $S$ is a non-klt centre of $(X,B)$ with $M|_{S}\equiv 0$.\\ 

\end{itemize}
Then there is a $\Q$-divisor $\Lambda\ge B$ such that  
\begin{itemize}
\item $(X,\Lambda)$ is lc over a neighbourhood of $z:=f(S)$, and

\item $n(K_X+\Lambda)\sim (n+2)M$.
\end{itemize}
\end{thm} 
 
This is a key ingredient of the proof of Theorem \ref{t-bnd-lct}. 
 Note that $K_X+\Lambda$ is 
actually a relative $n$-complement of $K_X+B$ over a neighbourhood of $z$ in the sense of [\ref{B-compl}, 2.18].
The important point here is that the complement is not an arbitrary one since $\Lambda$ is 
somehow controlled globally by $M$ as it satisfies the formula $n(K_X+\Lambda)\sim (n+2)M$.\\

We devote the rest of the introduction to rough sketches of the proofs of \ref{t-BAB}, \ref{t-bnd-lct-global}, 
and \ref{t-bnd-lct}.\\

{\textbf{\sffamily{Sketch of the proof of Theorem \ref{t-BAB}.}}} 
We will assume Theorem \ref{t-bnd-lct-global} in dimension $d$.
For simplicity we assume $B=0$.
The idea is to apply [\ref{B-compl}, Proposition 7.13], that is, \ref{t-from-lct-to-bnd-var} below. 
For this it is enough to show that 
there exist natural numbers $m,v$ and a positive real number $t$, all depending only on $d,\epsilon$, such that   
\begin{itemize}
\item $K_X$ has an $m$-complement, 

\item $|-mK_X|$ defines a birational map,

\item $\vol(-K_X)\le v$, and 

\item for any $0\le L\sim_\R -K_X$,  the pair $(X,tL)$ is klt.
\end{itemize}

The number $m$ is given by [\ref{B-compl}, Theorems 1.2 and 1.7]. 
The number $v$ is given by [\ref{B-compl}, Theorem 1.6] assuming \ref{t-BAB} in lower dimension. The number $t$ is given by 
Theorem \ref{t-bnd-lct-global}. In turn we will see that \ref{t-bnd-lct-global} is reduced 
to the general theorem on boundedness of lc thresholds, that is, \ref{t-bnd-lct}.

For reader's convenience we also include some rough explanations for as to why the existence of $m,v,t$ is enough to 
deduce boundedness of $X$. 
Applying [\ref{HX}], it is enough to show that $K_X$ has a klt $n$-complement 
for some bounded number $n\in\N$.
By assumption, there is an lc $m$-complement $K_X+B^+$.
If $X$ is exceptional,  the latter complement is automatically klt, so we are done in this case. 
Here by $X$ being exceptional we mean that $(X,L)$ is klt for every $0\le L\sim_\R -K_X$.

To treat the non-exceptional case the idea is to modify the complement 
$K_X+B^+$ into a klt one. 
We do this using birational boundedness. By induction we can assume that \ref{t-BAB} 
holds in lower dimension, so $X$ is birationally bounded by [\ref{B-compl}, Theorem 1.6]. 
In fact it turns out that $(X,B^+)$ is log birationally bounded, 
that is, there exists a log smooth projective pair $(V,\Sigma)$ belonging to a bounded family of pairs and  
there exists a birational map $V\bir X$ such that 
$\Sigma$ contains the exceptional divisors of $V\bir X$ and the support of 
the birational transform of $B^+$. 

Next we pull back $K_X+B^+$ to a high resolution of $X$ and push it down to $V$ 
and denote it by $K_V+B^+_V$. 
Then  $(V,B^+_V)$ is sub-lc and $m(K_V+B^+_V)\sim 0$. 
Now $\Supp B^+_V$ is contained in $\Sigma$.
So we can use the boundedness of $(V,\Sigma)$  
to perturb the coefficients of $ B^+_V$.  
 More precisely, perhaps after replacing $m$, there is $\Delta_V\sim_\Q B^+_V$ 
such that $(V,\Delta_V)$ is sub-klt and $m(K_V+\Delta_V)\sim 0$. 
Pulling back $K_V+\Delta_V$ to $X$ and denoting it by $K_X+\Delta$ we get a 
sub-klt pair $(X,\Delta)$ with $m(K_X+\Delta)\sim 0$.

Now a serious issue is that $\Delta$ may not be effective, so $K_X+\Delta$ is not necessarily an 
$m$-complement. In fact it is by no means clear that the coefficients of $\Delta$ are even bounded from below. 
However, it is not hard to see that existence of $t$ remedies the situation: indeed, 
if $0\le L\sim_\R -K_X$, then the coefficients of $L$ are bounded from above. This implies that 
the coefficients of $\Delta$ are bounded from below, by construction of $\Delta$. 
The rest of the argument which modifies  $\Delta$ to get a klt $n$-complement for some bounded $n$ 
is an easy application of 
the results of [\ref{B-compl}] on complements.\\

{\textbf{\sffamily{Sketch of the proof of Theorem \ref{t-bnd-lct-global}.}}} 
We will assume Theorem \ref{t-bnd-lct} in dimension $d$ and 
assume Theorem \ref{t-BAB} in lower dimension. Let $(X,B)$ and $A=-(K_X+B)$ be as in Theorem \ref{t-bnd-lct-global} 
in dimension $d$. Replacing $X$ with can assume it is $\Q$-factorial. 
Pick $\epsilon'\in (0,\epsilon)$ and pick $L\in|A|_\R$. 
Let $s$ be the largest number such that $(X,B+sL)$ is $\epsilon'$-lc.
It is enough to show $s$ is bounded from below away from zero. 

There is a prime divisor $T$ over $X$ with  
$$
a(T,X,B+sL)=\epsilon'.
$$
If $T$ is not exceptional over $X$, then  we let $\phi\colon Y\to X$ be the identity morphism
but if $T$ is exceptional over $X$, 
then we let $\phi\colon Y\to X$ be the extremal birational contraction which extracts $T$. 
Let $L_Y=\phi^*L$. One shows  
$\mu_{T}sL_Y\ge \epsilon-\epsilon'$, hence that it is enough to show that $\mu_{T}L_Y$ is bounded from 
above. 

Running an MMP on $-T$, restricting to the general fibres of the resulting Mori fibre 
space, and applying induction on dimension 
reduces the problem to the situation in which $X$ is an $\epsilon$-lc $\Q$-factorial Fano variety with 
Picard number one on which we want to show that $\mu_TL$ is bounded from above for any $L\in|-K_X|_\R$ and any 
prime divisor $T$ on $X$. Using the Picard number one property, we can replace $L$ and assume $\Supp L=T$, hence we can 
assume $L=uT$ with $u=\mu_TL$.

Applying  [\ref{B-compl}, Theorems 1.2, 1.6, 1.7] and Theorem \ref{t-BAB} in lower dimension 
we find a bounded number $n\in\N$ such that $K_X$ has an $n$-complement $K_X+\Omega$, $|-nK_X|$ defines a birational map 
and that $\vol(-K_X)$ is bounded from above. In particular, we deduce that
$(X,\Omega)$ is log birationally bounded. So there is a projective log smooth pair $(V,\Sigma)$ belonging to  
a bounded family and a birational map $X\bir V$ so that $\Sigma$ is reduced whose support contains the exceptional 
divisors of $V\bir X$ and the birational transform of $\Supp \Omega$. 

Pull back $K_X+B,L$ to a high resolution of $X$ and then push down to $V$ and denote the resulting divisors 
by $K_V+B_V,M$.
The main idea of the rest of the proof is to find a boundary $\Delta$ by taking an appropriate 
average between $B_V$ and $\Sigma$ so that 
\begin{itemize}
\item $(V,\Delta)$ is $\epsilon''$-lc for some fixed $\epsilon''>0$, 
\item $(V,\Delta+\frac{1}{u} M)$ is not klt, and 
\item ``degrees" of $\Delta$ and $M$ are bounded with respect to some very ample divisor.
\end{itemize} 
Applying Theorem \ref{t-bnd-lct} gives a 
positive lower bound for $\frac{1}{u}$, hence an upper bound for $u$.\\

{\textbf{\sffamily{Sketch of the proof of Theorem \ref{t-bnd-lct}.}}} 
It is easy to see that 
$$
\lct(X,B,|M|_\R)\ge \lct(X,B,|A|_\R),
$$
so we only need to find a positive lower bound for $\lct(X,B,|A|_\R)$. 
Pick $0\le N\sim_\R A$. 
Let $s$ be the largest number such that $(X,B+sN)$ is $\epsilon'$-lc where $\epsilon'=\frac{\epsilon}{2}$.
It is enough to show $s$ is bounded from below away from zero.   
There is a prime divisor $T$ on birational models of $X$  with log discrepancy 
$$
a(T,X,\Delta:=B+sN)=\epsilon'.
$$  
It is enough to show the multiplicity of $T$ in $\phi^*N$ is bounded from above 
on some resolution $\phi\colon V\to X$ on which $T$ is a divisor. 
We can assume the image of $T$ on $X$ is a closed point $x$ otherwise 
we can cut by hyperplane sections and apply induction on dimension. 
Since 
$$
a(T,X,\Delta)=\epsilon'<1,
$$ 
there is a birational contraction $Y\to X$ extracting $T$ but no other divisors. 
Moreover, we can assume that $-(K_Y+T)$ is ample over $X$, and using ACC for lc thresholds [\ref{HMX2}]
we can assume $(Y,T)$ is lc. 

The next step is to do a ``toroidalisation".
A key ingredient here is provided by the theory of complements, that is, Theorem \ref{t-compl-near-lcc}.
Using ampleness of $-(K_Y+T)$ over $X$, we can find  $\Lambda_Y$ such that 
$(Y,\Lambda_Y)$ is lc near $T$ and $n(K_Y+\Lambda_Y)\sim 0/X$ for a bounded $n\in\N$. 
Crucial point: if $\Lambda$ is the pushdown of $\Lambda_Y$, then after some delicate work we can assume 
$A-\Lambda$ is ample. In particular, the log discrepancy $a(T,X,\Lambda)=0$ and 
$(X,\Supp \Lambda)$ is log bounded. 
Using resolution of singularities we can  assume $(X,\Lambda)$ is log smooth and $\Lambda$ is reduced. 
The advantage of having $\Lambda$ is that now $T$ can be obtained by a sequence of blowups, toroidal 
 with respect to  $(X,\Lambda)$. 
The first step of this sequence is just the blowup of $x$.
One argues that it is enough to bound the number of these blowups.

We can discard any component of $\Lambda$ not passing through $x$, say $\Lambda=S_1+\dots+S_d$. 
A careful analysis of $Y\to X$ allows us to modify the situation so that 
$\Supp \Delta$ does not contain any stratum of $(X,\Lambda)$ apart from $x$. 
This is one of the difficult steps of the proof.

The next step is to do a ``torification".
Since $(X,\Lambda)$ is log smooth and log bounded, 
we can find a surjective finite morphism $X\to \mathbb{P}^d$ such that 
it maps $x$ to the origin 
$$
z=(1:0:\cdots:0),
$$
and it maps $S_i$ onto $H_i$ where $H_1,\dots,H_d$ are the coordinate hyperplanes passing through 
$z$. 
Since $\Supp \Delta$ does not contain any stratum of $(X,\Lambda)$ apart from $x$, 
it is not hard to reduce the problem to a similar problem on $\mathbb{P}^d$. 
From now on we assume 
$X=\mathbb{P}^d$ and that $S_i$ are the coordinate hyperplanes. 
The point of this reduction is that now $(X,\Lambda)$ is not only toroidal but actually toric,
 and $-(K_X+\Lambda)$ is ample. 
In particular, we can modify $\Delta$ so that $K_X+\Delta$ is numerically trivial.

Let $W\to X$ be the sequence of blowups which obtains $T$.
Since the blowups are toric, $W$ is a toric variety.
 If $Y\to X$ is the birational morphism contracting $T$ only, as before, then $Y$ is also a toric variety.
Moreover, if $K_Y+\Delta_Y$ is the 
pullback of $K_X+\Delta$, then $(Y,\Delta_Y)$ is $\epsilon'$-lc and $K_Y+\Delta_Y$ is numerically trivial.
Running MMP on $-K_Y$ we get another toric variety $Y'$ 
which is Fano and  $\epsilon'$-lc. 
By the toric version of \ref{t-BAB} [\ref{A-L-Borisov}],  $Y'$ belongs to a bounded 
family. From this we can produce a klt $m$-complement 
$K_{Y'}+\Omega_{Y'}$ for some bounded $m\in\N$
 which induces a klt $m$-complement 
$K_{Y}+\Omega_{Y}$ which in turn gives a klt $m$-complement 
$K_X+\Omega$. In particular, $(X,\Omega)$ belongs to a bounded family as $m(K_X+\Omega)\sim 0$. 
This implies that $(X,\Omega+u\Lambda)$ is klt for some 
$u>0$ bounded from below away from zero. 
Finally an easy calculation shows that the multiplicity of $T$ in 
the pullback of $\Lambda$ on $W$ is bounded from above
 which in turn implies the number of blowups in $W\to X$ is bounded as required.\\

{\textbf{\sffamily{Acknowledgements.}}}
I would like to thank Florin Ambro, Ivan Cheltsov, Christopher Hacon, Jingjun Han, Yujiro Kawamata, Mihai P\v{a}un, 
Vyacheslav Shokurov, and Yanning Xu as well as the referees for their valuable comments. Thanks to 
Jungkai A. Chen and the National Taiwan University office of the National Center for Theoretical Sciences 
for hosting a workshop on this paper and [\ref{B-compl}] and thanks to the participants for their comments.
It is likely that I am forgetting other people who have given me helpful comments on the paper 
in the last few years: I would like to thank those as well.

\section{\bf Preliminaries}

All the varieties in this paper are defined over an algebraically closed field $k$ of characteristic zero
unless stated otherwise.

\subsection{Divisors}\label{ss-divisors}
 Let $X$ be a normal variety and $D=\sum d_iD_i$ be an $\R$-divisor where $D_i$ are the distinct 
irreducible components of $D$. We sometimes denote the coefficient $d_i$ by $\mu_{D_i}D$. 
Let $Y$ also be a normal variety and $\phi\colon X\bir Y$ be a
birational map whose inverse does not contract any divisor.
We often denote $\phi_*D$ by $D_Y$.

Now let $X$ be a normal projective variety of dimension $d$ and $A,D$ be $\R$-Cartier 
$\R$-divisors. We define the \emph{degree} of $D$ with respect to $A$ to be the intersection number $\deg_AD:=A^{d-1}D$ 
when $d>1$ and $\deg_AD:=\deg D$ if $d=1$ where $\deg D$ denotes the usual degree of divisors on curves.
If $A-D$ is pseudo-effective, then one easily sees that $\deg_AD\le \deg_AA$. 

We can similarly define $\deg_AD$ even if $D$ is not $\R$-Cartier, for example, when $A$ is an ample $\Q$-divisor 
because in this case we can express $A^{d-1}$ as a $1$-cycle in the smooth locus of $X$.

The volume of an $\R$-divisor $D$ on a normal projective variety $X$ of dimension $d$ is defined as 
$$
\vol(D)=\limsup_{m\to \infty} \frac{h^0(\rddown{mD})}{m^d/d!}. 
$$\

\subsection{Pairs and singularities}\label{ss-pairs}

A \emph{sub-pair} $(X,B)$ consists of a normal quasi-projective variety $X$ and an $\R$-divisor 
$B$ with coefficients in $(-\infty, 1]$ such that $K_X+B$ is $\R$-Cartier. If $B\ge 0$, we call $B$ a \emph{boundary} 
and call $(X,B)$ a \emph{pair}. 

Let $\phi\colon W\to X$ be a log resolution of  a sub-pair $(X,B)$. Let $K_W+B_W$ be the 
pullback of $K_X+B$. The \emph{log discrepancy} of a prime divisor $D$ on $W$ with respect to $(X,B)$ 
is defined as 
$$
a(D,X,B):=1-\mu_DB_W.
$$
We say $(X,B)$ is \emph{sub-lc} (resp. \emph{sub-klt})(resp. \emph{sub-$\epsilon$-lc}) if 
 $a(D,X,B)$ is $\ge 0$ (resp. $>0$)(resp. $\ge \epsilon$) for every $D$. This means that  
every coefficient of $B_W$ is $\le 1$ (resp. $<1$)(resp. $\le 1-\epsilon$). If $(X,B)$ is a pair, 
we remove the sub and just say it is lc (resp. klt)(resp. $\epsilon$-lc). 
Note that since $a(D,X,B)=1$ for most prime divisors, we necessarily have $\epsilon\le 1$.

Let $(X,B)$ be a pair. An \emph{non-klt place} of $(X,B)$ is a prime divisor $D$ over $X$, that is, on 
birational models of $X$ such that $a(D,X,B)\le 0$. A \emph{non-klt centre} is the image on 
$X$ of a non-klt place. When $(X,B)$ is lc, a non-klt place and a non-klt centre are also sometimes 
referred to as an \emph{lc place} and an \emph{lc centre}, respectively.

A \emph{log smooth} pair is a pair $(X,B)$ where $X$ is smooth and $\Supp B$ has simple 
normal crossing singularities. Assume $(X,B)$ is log smooth and assume $B=\sum_1^r B_i$ is reduced 
where $B_i$ are the irreducible components. 
A \emph{stratum} of $(X,B)$ is an irreducible component of $\bigcap_{i\in I}B_i$ for some $I\subseteq \{1,\dots,r\}$.  
Since $B$ is reduced, a stratum is nothing but an lc centre of $(X,B)$ or $X$ itself.

\begin{lem}\label{l-average-boundary}
If $(X,B)$ and $(X,B')$ are sub-pairs and $\Delta=tB+(1-t)B'$ for some real number $t\in[0,1]$, then 
$$
a(D,X,\Delta)=ta(D,X,B)+(1-t)a(D,X,B')
$$
for any prime divisor $D$ over $X$. 
In particular, if $(X,B)$ is sub-$\epsilon$-lc and $(X,B')$ is sub-$\epsilon'$-lc, then 
$(X,\Delta)$ is sub-$(t\epsilon+(1-t)\epsilon')$-lc.
\end{lem}

We leave the proof to the reader.

\subsection{Fano pairs}

A pair $(X,B)$ is called \emph{Fano} (resp. \emph{weak Fano}) if it is lc and $-(K_X+B)$ is ample 
(resp. nef and big). When $B=0$ we just say $X$ is Fano (resp. weak Fano). 
A variety $X$ is of \emph{Fano type} if $(X,B)$ is klt weak Fano for some $B$. 
By [\ref{BCHM}], a Fano type variety is a Mori dream space, so we can run an MMP 
on any $\R$-Cartier $\R$-divisor on $X$ and it terminates.

\subsection{Minimal models, Mori fibre spaces, and MMP}

Let $ X\to Z$ be a
projective morphism of normal varieties and $D$ be an $\R$-Cartier $\R$-divisor
on $X$. Let $Y$ be a normal variety projective over $Z$ and $\phi\colon X\bir Y/Z$
be a birational map whose inverse does not contract any divisor. 
Assume $D_Y:=\phi_*D$ is also $\R$-Cartier and that 
there is a common resolution $g\colon W\to X$ and $h\colon W\to Y$ such that
$E:=g^*D-h^*D_Y$ is effective and exceptional$/Y$, and
$\Supp g_*E$ contains all the exceptional divisors of $\phi$.

Under the above assumptions we call $Y$ 
a \emph{minimal model} of $D$ over $Z$ if $D_Y$ is nef$/Z$.
On the other hand, we call $Y$ a \emph{Mori fibre space} of $D$ over $Z$ if there is an extremal contraction
$Y\to T/Z$ with $-D_Y$  ample$/T$ and $\dim Y>\dim T$.

If one can run a \emph{minimal model program} (MMP) on $D$ over $Z$ which terminates 
with a model $Y$, then $Y$ is either a minimal model or a Mori fibre space of 
$D$ over $Z$. If $X$ is a Mori dream space, 
eg if $X$ is of Fano type over $Z$, then such an MMP always exists by [\ref{BCHM}].

\subsection{Plt pairs}\label{ss-plt-blowup}

In this subsection we construct plt models associated with certain lc pairs. This is 
similar to the construction of plt blowups [\ref{Prokhorov-plt-blowups}] 
(also see [\ref{PSh-I}, Definition 3.5] and [\ref{Xu-plt}, Lemma 1]).

\begin{lem}\label{l-plt-blowup}
Assume  $(X,B)$ is an lc pair. Assume that $(X,B)$ is not klt but $(X,C)$ is klt for some boundary $C$. 
Then there exist a prime divisor $T$ over $X$ and a projective birational morphism $\phi\colon Y\to X$ 
such that 
\begin{itemize}
\item either $\phi$ is small or it contracts $T$ but no other divisors,

\item $(Y,T)$ is plt,

\item $-(K_Y+T)$ is ample over $X$, and 

\item $a(T,X,B)=0$.
\end{itemize}
\end{lem}
\begin{proof}
Assume that there is a boundary $\Gamma$ and a prime divisor $T$ over $X$ such that 
$$
a(T,X,B)=a(T,X,\Gamma)=0
$$ 
and such that $(X,\Gamma)$ has no lc place other than $T$. Replacing $C$ with 
$(1-t)B+tC$ for a sufficiently small real number $t>0$, we can assume that $a(T,X,C)\le 1$, 
by Lemma \ref{l-average-boundary}. If $T$ is not exceptional over $X$, let $Y\to X$ be a small $\Q$-factorialisation. 
But if $T$ is exceptional over $X$, let $Y\to X$ be the 
birational contraction from $\Q$-factorial $Y$ which extracts $T$ but no other divisor. 
Let $K_Y+C_Y$ be the pullback of $K_X+C$. Then $(Y,C_Y)$ is klt.

Run an MMP on $-(K_Y+T)$ over $X$ and let $Y'$ be the resulting model; we can run such an MMP 
as $Y$ is of Fano type over $X$. 
Next let $Y'\to Y''/X$ be the contraction defined by $-(K_{Y'}+T')$. It turns out that $T$ 
is not contracted over $Y''$: indeed, otherwise $(Y'',0)$ would be lc but not klt because 
$K_{Y'}+T'\sim_\Q 0/Y''$, and this contradicts the fact that $(Y'',C_{Y''})$ is klt.   
Now replace $Y,T$ with $Y'',T''$ where $T''$ is the pushdown of $T$ to $Y''$. 
By construction, $-(K_Y+T)$ is ample over $X$. 
Moreover, if $K_Y+\Gamma_Y$ is the pullback of $K_X+\Gamma$, then 
$(Y,\Gamma_Y)$ is plt because $a(D,Y,\Gamma_Y)>0$ for every prime divisor $D$ over $X$ 
other than $T$. So $(Y,T)$ is plt and $Y\to X$ is the desired morphism. 

We will find $\Gamma,T$ as in the first paragraph.
 Let $\psi\colon W\to X$ be a log resolution of $(X,B+C)$ and write 
$$
K_W+B_W=\psi^*(K_X+B) ~~~\mbox{and} ~~~ K_W+C_W=\psi^*(K_X+C). 
$$ 
Write 
$$
B_W-C_W\sim_\R H_W+G_W/X
$$ 
where $H_W\ge 0$ is ample over $X$ and $G_W\ge 0$. 
Let $H,G$ be the pushdowns of $H_W,G_W$ to $X$. Note that $H_W+G_W\sim_\R 0/X$, so $H+G$ is $\R$-Cartier. 
Replacing $\psi$ with a higher resolution and replacing $H_W,G_W$ appropriately 
we can assume that $\psi$ is a log resolution of 
$$
(X,B+C+H+G).
$$ 
Moreover, replacing $C$ with $(1-t)B+tC$ for a sufficiently small real number 
$t>0$ we can assume that the coefficients of $B_W-C_W$ are sufficiently close to $0$.

Let $E$ be the sum of the components of $B_W$ with coefficient $1$. Since $(X,B)$ is not klt, $E\neq 0$. 
First assume that $E$ and $G_W$ have no common component. Then 
we can pick a small real number $a>0$ and find 
$$
\Gamma_W\sim_\R B_W+aH_W+aG_W
$$
so that $(W,\Gamma_W)$ is sub-lc, $(W,\Supp \Gamma_W)$ is log smooth, 
$\mu_T\Gamma_W=1$ for some component $T$ of $E$, all the other coefficients of $\Gamma_W$ are $<1$, 
and $\Gamma:=\psi_*\Gamma_W$ is effective. Since $K_Y+\Gamma_Y\sim_\R 0/X$, the pair 
$(X,\Gamma)$ is lc with the unique lc place $T$. 
So we are done in this case. 

Now assume $G_W$ and $E$ have common components.
Recall that the coefficient $\mu_TC_W$ of each component $T$ of $E$ is sufficiently close to $1$. 
Let $a>0$ be the largest real number such that $(W,C_W+aG_W)$ is sub-lc. Then we 
can assume that each component of $C_W+aG_W$ with coefficient $1$ is a component of $E$. 
Thus we can find 
$$
\Gamma_W\sim_\R C_W+aH_W+aG_W
$$
so that $(W,\Gamma_W)$ is sub-lc, $(W,\Supp \Gamma_W)$ is log smooth, 
$\mu_T\Gamma_W=1$ for some component $T$ of $E$, all the other coefficients of $\Gamma_W$ are $<1$, 
and $\Gamma:=\psi_*\Gamma_W$ is effective. Since $K_Y+\Gamma_Y\sim_\R 0/X$, the pair 
$(X,\Gamma)$ is lc with the unique lc place $T$. 
So we are again done. 

\end{proof}

\subsection{Bounded families of pairs}\label{ss-bnd-couples}
We say a set $\mathcal{Q}$ of normal projective varieties is \emph{birationally bounded} (resp. \emph{bounded}) 
if there exist finitely many projective morphisms $V^i\to T^i$ of varieties  
such that for each $X\in \mathcal{Q}$ there exist an $i$, a closed point $t\in T^i$, and a 
birational isomorphism (resp. isomorphism) $\phi\colon V^i_t\bir X$  where $V_t^i$ is the fibre of $V^i\to T^i$ over $t$.

Next we will define boundedness for couples. 
A \emph{couple} $(X,S)$ consists of a normal projective variety $X$ and a  divisor 
$S$ on $X$ whose coefficients are all equal to $1$, i.e. $S$ is a reduced divisor. 
We use the term couple instead of pair because $K_X+S$ is not assumed to be $\Q$-Cartier and 
$(X,S)$ is not assumed to have good singularities.
 
We say that a set $\mathcal{P}$ of couples  is \emph{birationally bounded} if there exist 
finitely many projective morphisms $V^i\to T^i$ of varieties and reduced divisors $C^i$ on $V^i$ 
such that for each $(X,S)\in \mathcal{P}$ there exist an $i$, a closed point $t\in T^i$, and a 
birational isomorphism $\phi\colon V^i_t\bir X$ such that $(V^i_t,C^i_t)$ is a couple and 
$E\le C_t^i$ where $V_t^i$ and $C_t^i$ are the fibres over $t$ of the morphisms $V^i\to T^i$ 
and $C^i\to T^i$, respectively, and $E$ is the sum of the 
birational transform of $S$ and the reduced exceptional divisor of $\phi$.
We say $\mathcal{P}$ is \emph{bounded} if we can choose $\phi$ to be an isomorphism. 
  
A set $\mathcal{R}$ of projective pairs $(X,B)$ is said to be {log birationally bounded} (resp. log {bounded}) 
if the set of the corresponding couples $(X,\Supp B)$ is birationally bounded (resp. bounded).
Note that this does not put any condition on the coefficients of $B$, eg we are not requiring the 
coefficients of $B$ to be in a finite set.

\subsection{Effective birationality and birational boundedness}

In the next few subsections, we recall some of the main results of [\ref{B-compl}] 
which are needed in this paper.

\begin{thm}[{[\ref{B-compl}, Theorem 1.2]}]\label{t-eff-bir-e-lc}
Let $d$ be a natural number and $\epsilon$ be a positive real number. Then there is a natural 
number $m$ depending only on $d$ and $\epsilon$ such that if $X$ is any $\epsilon$-lc weak Fano 
variety of dimension $d$, then $|-mK_X|$ defines a birational map.
\end{thm}

\begin{thm}[{[\ref{B-compl}, Theorem 1.6]}]\label{t-BAB-to-bnd-vol}
Let $d$ be a natural number and $\epsilon$ be a positive real number. 
Assume Theorem \ref{t-BAB} holds in dimension $d-1$. Then there is a 
number $v$ depending only on $d$ and $\epsilon$ such that if $X$ is an $\epsilon$-lc weak Fano variety 
of dimension $d$, then the volume $\vol(-K_X)\le v$. In particular, such $X$ are birationally bounded. 
\end{thm}

In fact the proof of the theorem shows that $(X,0)$ is log bounded.\

\subsection{Complements}\label{ss-complements}

Let $(X,B)$ be a projective pair. Let $T=\rddown{B}$ and $\Delta=B-T$. 
An \emph{$n$-complement} of $K_{X}+B$  is of the form $K_{X}+{B}^+$ where

$\bullet$ $(X,{B}^+)$ is lc, 

$\bullet$ $n(K_{X}+{B}^+)\sim 0$, and 

$\bullet$ $n{B}^+\ge nT+\rddown{(n+1)\Delta}$.\\

\begin{thm}[{[\ref{B-compl}, Theorem 1.7]}]\label{t-bnd-compl}
Let $d$ be a natural number and $\mathfrak{R}\subset [0,1]$ be a finite set of rational numbers.
Then there exists a natural number $n$ depending only on $d$ and $\mathfrak{R}$ satisfying the following.  
Assume $(X,B)$ is a projective pair such that 

$\bullet$ $(X,B)$ is lc of dimension $d$,

$\bullet$ $B\in \Phi(\mathfrak{R})$, that is, the coefficients of $B$ are in $\Phi(\mathfrak{R})$, 

$\bullet$ $X$ is of Fano type, and 

$\bullet$ $-(K_{X}+B)$ is nef.\\\\
Then there is an $n$-complement $K_{X}+{B}^+$ of $K_{X}+{B}$ 
such that ${B}^+\ge B$. Moreover, the complement is also an $mn$-complement for any $m\in \N$.\\ 
\end{thm}

In the theorem 
$$
\Phi(\mathfrak{R}):=\{1-\frac{r}{m} \mid r\in \mathfrak{R}, m\in \N\}.
$$ \

\subsection{From bounds on lc thresholds to boundedness of varieties}

The following result connects lc thresholds and boundedness of Fano varieties, and it is one of the 
main ingredients of the proof of Theorem \ref{t-BAB}.

\begin{thm}[{[\ref{B-compl}, Proposition 7.13]}]\label{t-from-lct-to-bnd-var}
Let $d,m,v$ be natural numbers and let $t_l$ be a sequence of positive real numbers. 
Let $\mathcal{P}$ be the set of projective varieties $X$ such that 

$\bullet$ $X$ is a klt weak Fano variety of dimension $d$, 

$\bullet$ $K_X$ has an $m$-complement, 

$\bullet$ $|-mK_X|$ defines a birational map,

$\bullet$ $\vol(-K_X)\le v$, and 

$\bullet$ for any $l\in\N$ and any $L\in |-lK_X|$,  the pair $(X,t_lL)$ is klt.\\\\
Then $\mathcal{P}$ is a bounded family.
\end{thm}

Assuming Theorem \ref{t-BAB} in lower dimension, Theorems \ref{t-eff-bir-e-lc}, \ref{t-BAB-to-bnd-vol}, and \ref{t-bnd-compl} show that all the assumptions 
of \ref{t-from-lct-to-bnd-var} are satisfied for $X$ as in \ref{t-BAB} in dimension $d$ (when $B=0$) except 
the last assumption. We will use Theorem \ref{t-bnd-lct-global} to show that this last assumption is 
also satisfied.

The theorem also has applications to boundedness of $K$-semistable Fano varieties [\ref{Jiang-3}, Theorem 1.1].

\subsection{Sequences of blowups}\label{ss-blowup-seq}
We discuss some elementary aspects of blowups.

(1)
Let $X$ be a smooth variety of dimension $\ge 2$ and let 
$$
\cdots \to X_{i+1}\to X_i\to \cdots \to X_0=X
$$ 
be a (finite or infinite) sequence of smooth blowups,  that is, each $X_{i+1}\to X_{i}$ 
is the blowup along a smooth subvariety $C_i$ of codimension $\ge 2$. 
If the sequence is finite, say $X_p\to X_{p-1}$ is the last blowup,  
the \emph{length} of the sequence is $p$. We denote the exceptional divisor of 
$X_{i+1}\to X_{i}$ by $E_{i+1}$.

(2)
Let $\Lambda$ be a reduced divisor such that $(X,\Lambda)$ is log smooth. 
We say a sequence as in (1) is \emph{toroidal} with respect to $(X,\Lambda)$, if 
for each $i$, the centre $C_i$ is a stratum of $(X_i,\Lambda_i)$ where $K_{X_i}+\Lambda_i$ is the pullback of 
$K_X+\Lambda$ (cf. [\ref{Kollar-toroidal}]). This is equivalent to saying that 
the exceptional divisor of each blowup in the sequence is an lc place of $(X,\Lambda)$.  

\begin{lem}\label{l-coeff-tor-blowup}
Under the above notation, assume we have a finite sequence of smooth blowups of length $p$, 
toroidal with respect to $(X,\Lambda)$, and let 
$$
\phi\colon X_p\to X_0=X
$$ 
be the induced morphism. 
Suppose $C_i\subset E_{i}$ for each $0<i<p$. Then $\mu_{E_p} \phi^*\Lambda\ge p+1$.
\end{lem}
\begin{proof}
 If $0\le i<p$, then $C_i$ is contained in at least two components of 
 $\Lambda_i$ because $C_i$ is an lc centre of $(X_i,\Lambda_i)$ of codimension $\ge 2$. 
 When $i>0$, one of these components is $E_i$, by assumption. 
  Let $\phi_i$ denote $X_i\to X_0$. Then by the equality $\Lambda_i=\Supp \phi_{i}^*\Lambda$ 
  and by induction on $i$ we have 
 $$
 \mu_{E_{i+1}}\phi_{i+1}^*\Lambda\ge \mu_{E_i}\phi_{i}^*\Lambda+1\ge i+2.
 $$ 

\end{proof}

(3) 
Consider a sequence of blowups as in (1) (so this is not necessarily toroidal). 
Let $T$ be a prime divisor over $X$, that is, on birational models of $X$.  
Assume that for each $i$,  $C_i$ is the centre of $T$ on  $X_i$. We then call the 
sequence a \emph{sequence of centre blowups} associated to $T$.
By [\ref{kollar-mori}, Lemma 2.45], such a sequence cannot be infinite, that is, 
after finitely many centre blowups, $T$ is obtained, i.e. there is $p$ such that  
$T$ is the exceptional divisor of $X_p\to X_{p-1}$ 
(here we think of $T$ birationally; if $T$ is fixed on some model, then we should say the exceptional divisor is the 
birational transform of $T$). In this case, we say \emph{$T$ is obtained by the 
sequence of centre blowups} $X_{p}\to X_{p-1}\to \cdots \to X_0=X$.

\subsection{Analytic pairs and analytic neighbourhoods of algebraic singularities}\label{ss-anal-pairs}
For convenience we will recall certain analytic notions shortly. We will use these mainly to 
compare analytic neighbourhoods of algebraic singularities and this involves only elementary aspects of the analytic theory. 
Strictly speaking we can replace these by 
purely algebraic constructions, eg formal varieties and formal neighbourhoods, but we prefer the analytic language as it 
is more straightforward. When we have an algebraic object $A$ defined over $\C$, eg a variety, a morphism, etc, we denote 
the associated analytic object by $A^{\rm an}$.

(1) 
An \emph{analytic pair} $(U,G)$ consists of a normal complex analytic variety $U$ and an 
 $\R$-divisor $G$ with finitely many components and with coefficients in $[0,1]$ such that $K_U+G$ is $\R$-Cartier. 
\emph{Log discrepancies} and notions of \emph{lc, klt, $\epsilon$-lc} singularities can be 
defined for such pairs just as in \ref{ss-pairs} using log resolutions. In this paper, 
we will only need analytic pairs $(U,G)$ which are derived from 
 algebraic pairs with $U$ being smooth. 

Two analytic pairs $(U,G)$ and $(U',G')$ are \emph{analytically isomorphic} 
if  there is an analytic isomorphism, that is, a biholomorphic map $\nu\colon U\to U'$ such that $\nu_*G=G'$.

(2)
Let $(X,B)$ be an algebraic pair over $\C$ (that is, a pair as in \ref{ss-pairs}),  
let $x\in X$ be a closed point, and let $U$ be an analytic neighbourhood of $x$ in the associated 
analytic variety $X^{\rm an}$. 
Take a log resolution $\phi\colon W\to X$ and 
let $V$ be the inverse image of $U$ under 
$\phi^{\rm an}\colon W^{\rm an}\to X^{\rm an}$. Then $\phi^{\rm an}|_V$ is an analytic log resolution 
of $(U,B^{\rm an}|_U)$. In particular, if $(U,B^{\rm an}|_U)$ is $\epsilon$-lc in the analytic sense, then 
$(X,B)$ is $\epsilon$-lc in some algebraic neighbourhood of $x$. Conversely, it is clear that if 
$(X,B)$ is $\epsilon$-lc in some algebraic neighbourhood of $x$, then we can choose $U$ so that 
$(U,B^{\rm an}|_U)$ is $\epsilon$-lc in the analytic sense.

(3)
Now let $(X,B)$ and $(X',B')$ be algebraic pairs over $\C$,  
let $x\in X$ and $x'\in X'$ be closed points, and let $U$ and $U'$ be analytic neighbourhoods 
of $x$ and $x'$, respectively.  Assume that $(U,B^{\rm an}|_U)$ and $(U',B'^{\rm an}|_{U'})$ are analytically 
isomorphic. Then $(X,B)$ is $\epsilon$-lc in some algebraic neighbourhood of $x$ 
if and only if $(X',B')$ is $\epsilon$-lc in some algebraic neighbourhood of $x'$. 
Note that it may well happen that $(X,B)$ and $(X',B')$ are not algebraically isomorphic 
in any algebraic neighbourhoods of $x$ and $x'$.
 For example, the two pairs $(\PP^2,B)$ and $(\PP^2,B')$ 
have isomorphic analytic neighbourhoods where $B$ is an irreducible curve 
with a node at $x$ but $B'$ is the union of two lines intersecting at $x'$.

\subsection{\'Etale morphisms}

We look at singularities of images of a pair under a finite morphism which is \'etale at some point.

\begin{lem}\label{l-etale-0}
Let $(X,B=\sum b_jB_j)$ be a pair over $\C$, $\pi\colon X\to Z$ be a finite morphism, 
$x\in X$ a closed point, and $z=\pi(x)$.  Assume 

$\bullet$ $X$ and $Z$ are smooth near $x$ and $z$, respectively,

$\bullet$ $\pi$ is \'etale at $x$, 

$\bullet$ $\Supp B$ does not contain any point of $\pi^{-1}\{z\}$ except possibly $x$, and 

$\bullet$ $C:=\pi(B):=\sum b_j\pi(B_j)$.\\\\
Then $(X,B)$ is $\epsilon$-lc near $x$ if and only if $(Z,C)$ is $\epsilon$-lc near $z$. More precisely, 
there exist analytic neighbourhoods $U$ and $V$ of $x$ and $z$, respectively, such that 
$\pi^{\rm an}|_U$ induces an analytic isomorphism between $(U,B^{\rm an}|_U)$ and $(V,C^{\rm an}|_V)$.
\end{lem}

We leave the proof of this lemma and the next lemma to the reader. 

Note that in the previous lemma, $C$ may not even be a boundary away from $z$, that is, it may have 
components not passing through $z$ but with coefficients larger than $1$. 

The next lemma is useful for showing that a morphism is \'etale at a point.

\begin{lem}\label{l-etale}
Let $\pi\colon X\to Z$ be a finite morphism between varieties of dimension $d$, 
$x\in X$ a closed point, and $z=\pi(x)$. Assume $X$ and $Z$ are smooth at $x$ and $z$, respectively.
Assume $t_1,\dots,t_d$ are local parameters at $z$ and that $\pi^*t_1,\dots, \pi^*t_d$ are local 
parameters at $x$. Then $\pi$ is \'etale at $x$.
\end{lem}\

\subsection{Toric varieties and toric MMP}

We will reduce Theorem \ref{t-bnd-lct} to the case when $X=\PP^d$. To deal with this case we 
need some elementary toric geometry. All we need can be found in [\ref{toric-cox-etal}]. 
Let $X$ be a (normal) $\Q$-factorial projective toric variety. Then $X$ is a Mori dream space, meaning we can run an MMP 
on any $\Q$-divisor $D$ which terminates with a minimal model or a Mori fibre space of $D$. 
Moreover, the MMP is toric, that is, all the contractions and varieties in the process are toric. 
If we have a projective toric morphism $X\to Z$ to a toric variety, then we can run an MMP on $D$ over $Z$ 
which terminates with a minimal model or a Mori fibre space of $D$ over $Z$. See 
[\ref{toric-cox-etal}, \S 15.5] for proofs. 

Now let $\Lambda$ be the sum of some of the torus-invariant divisors on a projective toric variety $X$, and 
assume $(X,\Lambda)$ is log smooth. Let $Y\to X$ be a sequence of blowups toroidal with respect to 
$(X,\Lambda)$. Then $Y$ is also a toric variety as each blowup in the process is a blowup along an 
orbit closure.

\subsection{Bounded small modifications}
The following lemma is useful for reducing problems to the case when the canonical divisor is $\Q$-Cartier.

\begin{lem}\label{l-bnd-small-modification}
Let $d,r$ be natural numbers and $\epsilon$ be a positive real number. 
Then there is a natural number $l$ depending only 
on $d,r,\epsilon$ satisfying the following. Assume that 
\begin{itemize}
\item $(X,B)$ is a projective $\epsilon$-lc pair of dimension $d$, and 

\item $A$ is a very ample divisor on $X$ with $A^d\le r$.
\end{itemize}
 Then there is a projective small birational morphism $\phi\colon Y\to X$ and a very ample divisor $A_Y$ on $Y$ such that 
\begin{itemize}
\item $Y$ is normal, 

\item $K_Y$ is $\Q$-Cartier, 

\item $A_Y^d\le l$, and 

\item $A_Y-\phi^*A$ is ample. 
\end{itemize} 
\end{lem}

\begin{proof}
Since $A$ is very ample and $A^d\le r$, $X$ belongs to a bounded set of projective varieties. 
Then there exist finitely many projective morphisms $V^i\to T^i$ of quasi-projective 
varieties such that each $X$ in the lemma is isomorphic to the fibre of $V^i\to T^i$ over 
some closed point $t\in T^i$ for some $i$. We can also assume that there is a divisor $A_V^i$ on $V^i$ 
which is very ample over $T^i$ and such that $A\sim A_V^i|_X$.
For the rest of the proof we fix $i$ and write $V=V^i, T=T^i, A_V=A_V^i$. We will shrink 
$V,T$ when convenient.
 
Let $W\to V$ be a log resolution. Let $\Sigma$ be the reduced exceptional  
divisor of $W\to V$. 
Let $\Gamma_W=(1-\frac{\epsilon}{2})\Sigma$. Run an MMP on $K_W+\Gamma_W$ over $V$ and let $(U,\Gamma_U)$ be the 
resulting log minimal model of $(W,\Gamma_W)$ over $V$. 
Shrinking $W,U,V,T$ we 
can assume that the following holds: if $t\in T$ is a closed point 
and if $Y,X$ are the fibres of $U\to T$ and $V\to T$ over $t$, respectively, 
then 
\begin{itemize}
\item $Y$ is normal, 

\item $K_Y\sim_\Q K_U|_Y$,

\item $\Gamma_U$ does not contain $Y$,

\item support of $\Gamma_Y:=\Gamma_U|_Y$ coincides with the reduced exceptional divisor of $Y\to X$, 

\item and each non-zero coefficient of $\Gamma_Y$ is equal to $1-\frac{\epsilon}{2}$. 
\end{itemize}

Now let $(X,B)$ be as in the lemma. We can assume $X$ is isomorphic to the fibre of $V\to T$ over some closed point $t$.
Let $Y$ be the fibre of $U\to T$ over $t$ and let $\phi\colon Y\to X$ be the induced birational morphism. 
We show that $\phi$ is a small morphism. Let 
$K_Y+B_Y=\phi^*(K_X+B)$. Then 
$$
B_Y-\Gamma_Y=(K_Y+B_Y)-(K_Y+\Gamma_Y)\equiv -(K_Y+\Gamma_Y)/X
$$ 
is anti-nef over $X$. Moreover, $\phi_*(B_Y-\Gamma_Y)=B\ge 0$. Thus by the negativity lemma, 
$B_Y- \Gamma_Y\ge 0$. However, $(X,B)$ is $\epsilon$-lc, so each coefficient of $B_Y$ is $\le 1-\epsilon$ 
but by the previous paragraph each non-zero coefficient of $\Gamma_Y$ is $\ge 1-\frac{\epsilon}{2}$. 
This is possible only if $\Gamma_Y=0$. Therefore, 
$\phi$ is a small morphism as $\Supp \Gamma_Y$ contains every divisor contracted by $\phi$. 
Moreover, since $U$ is $\Q$-factorial, $K_Y\sim_\Q K_U|_Y$ is $\Q$-Cartier. 

Pick a divisor $A_U$ on $U$ which is very ample over $T$. Let $A_Y=A_U|_Y$. Then $A_Y$ is very ample on $Y$ 
and $A_Y^d\le l$ for some fixed natural number $l$. 
We can assume that $A_U-\psi^*A_V$ is ample over $T$ where $\psi$ is the morphism $U\to V$.
This ensures that $A_Y-\phi^*A$ is ample. 

\end{proof}

\subsection{Semi-ample divisors}

\begin{lem}\label{l-semi-ample-div}
Assume $Y\to X$ is a contraction of normal projective varieties, $C$ is a nef $\R$-divisor on 
$Y$ and $A$ is the pullback of an ample $\R$-divisor on $X$. If $C$ is semi-ample over $X$, then 
$C+aA$ is semi-ample (globally) for any real number $a>0$.
\end{lem}
\begin{proof}
Since $C$ is semi-ample over $X$, it defines a contraction $\phi \colon Y\to Z/X$ to a normal projective variety. Replacing $Y$ 
with $Z$ and replacing $C,A$ with $\phi_*C,\phi_*A$, respectively, we can assume $C$ is ample over $X$. Pick $a>0$. 
Now $C+bA$ is ample for some $b\gg a$ because 
$A$ is the pullback of an ample divisor on $X$. Since $C$ is globally nef, 
$$
C+tbA=(1-t)C+t(C+bA)
$$ 
is ample for any $t\in (0,1]$. In particular, taking $t=\frac{a}{b}$ we see that $C+aA$ is ample.

\end{proof}


\section{\bf Lc thresholds of anti-log canonical systems of Fano pairs}

In this section we study lc thresholds of the $\R$-linear systems $|-(K_X+B)|_\R$ for (weak) log Fano pairs $(X,B)$. 
As pointed out in the introduction, understanding such thresholds is key to the proof of Theorem \ref{t-BAB}. 
We will show that Theorem \ref{t-bnd-lct} implies Theorem \ref{t-bnd-lct-global} assuming 
Theorem \ref{t-BAB} in lower dimension. We also prove that Theorem \ref{t-bnd-lct} implies 
Theorem \ref{t-global-lct-attained}. First we consider a special case of Theorem \ref{t-bnd-lct-global}.

\begin{prop}\label{p-bnd-lct-global-weak}
Let $d$ be a natural number and $\epsilon$ be a positive real number. 
Assume  that Theorem \ref{t-bnd-lct} holds  in dimension $\le d$ and that Theorem \ref{t-BAB}
holds in dimension $\le d-1$. Then there is a 
positive real number $v$ depending only on $d,\epsilon$ satisfying the following. 
Assume that 

$\bullet$ $X$ is a $\Q$-factorial $\epsilon$-lc Fano variety of dimension $d$,

$\bullet$ $X$ has Picard number one, and

$\bullet$ $0\le L\sim_\R -K_X$.\\\\
Then each coefficient of $L$ is less than or equal to $v$. 
\end{prop}
\begin{proof}
\emph{Step 1.}
In this step we do some preparations.
Pick a component $T$ of $L$. Since $X$ has Picard number one, $L\sim_\R uT$ for some $u\ge \mu_TL$.  
Thus we may replace $L$, hence assume $L=u T$. 
We need to show $u$ is bounded from above. 
By Theorem \ref{t-bnd-compl}, there is a natural number $n$ depending only on $d$ such that 
$K_X$ has an $n$-complement $K_X+\Omega$. 
Moreover, by Theorems \ref{t-eff-bir-e-lc} and \ref{t-BAB-to-bnd-vol} in dimension $d$, and 
\ref{t-BAB} in lower dimension, 
replacing $n$ depending only on $d,\epsilon$, we can assume $|-nK_X|$ defines a birational map 
and that $\vol(-K_X)$ is bounded from above.\\

\emph{Step 2.}
In this step we show that $(X,\Omega)$ is log birationally bounded. Indeed 
applying [\ref{B-compl}, Proposition 4.4] (by taking $B=0$ and $M=n\Omega$) we deduce that 
there exist a number $c \in \R^{>0}$ and a 
bounded set of couples $\mathcal{P}$ depending only on $d,\epsilon$  satisfying the following: 
there is a projective log smooth couple $(V,\Lambda)\in \mathcal{P}$ 
and a birational map $V\bir X$ such that 
\begin{itemize}
\item  $\Supp \Lambda$ contains the exceptional 
divisor of  $V\bir X$ and the birational transform of $\Supp \Omega$;

\item  for a common resolution $\phi\colon W\to X$ and $\psi\colon W\to V$, 
each coefficient of $\psi_*\phi^*\Omega$ is at most $c$.

\end{itemize}

Enlarging $\Lambda$ we can also assume that $H\le \Lambda$ for some very ample divisor $H\ge 0$.\\

\emph{Step 3.}
Let $B$ be a boundary such that $(X,B)$ is $\epsilon$-lc and $K_X+B\sim_\R 0$. Let 
$$
K_V+B_V=\psi_*\phi^*(K_X+B) ~~\mbox{and}~~ K_V+\Omega_V=\psi_*\phi^*(K_X+\Omega).
$$
Then  $(V,B_V)$ is sub-$\epsilon$-lc and 
$$
a(T,V,B_V)=a(T,X,B)\le 1.
$$
Similarly,  
$(V,\Omega_V)$ is sub-lc and 
$$
a(T,V,\Omega_V)=a(T,X,\Omega)\le 1.
$$ 
By Step 2,  the union of $\Supp \Omega_V$ and the exceptional divisors of $V\bir X$ 
is contained in $\Supp \Lambda$, hence $\Omega_V\le \Lambda$ 
which implies 
$$
a(T,V,\Lambda)\le a(T,V,\Omega_V)\le 1. 
$$\

\emph{Step 4.}
In this step we show that the coefficients of $B_V$ are bounded from below.
Assume $D$ is a component of $B_V$ with negative coefficient. 
Let $K_V+\Gamma_V=\psi_*\phi^*K_X$. Then $\Gamma_V+\psi_*\phi^*B=B_V$, hence 
$\mu_D\Gamma_V\le \mu_DB_V$, so it is enough to bound $\mu_D\Gamma_V$ from below. 
Now $K_V+\Gamma_V\equiv -\psi_*\phi^*\Omega$, hence $\deg_H(K_V+\Gamma_V)$ 
is bounded from below, by Step 2. Thus $\deg_H\Gamma_V$ is bounded from below 
because $\deg_HK_V$ belongs to a fixed finite set as 
$(V,\Lambda)\in\mathcal{P}$ and $H\le \Lambda$.

Write  $\Gamma_V=I-J$ where $I,J$ are effective divisors with no common components. 
Since $I\le \Lambda$, $\deg_HI\le \deg_H\Lambda$ which shows $\deg_HI$ is bounded from 
above, hence $\deg_HJ$ is bounded from above too. Therefore, the coefficients of $J$ are 
bounded from above which in turn shows the coefficient of $D$ in $\Gamma_V$ is bounded from below as required. 
Note that this also implies the coefficients of $\Omega_V$ are bounded from below.\\ 

\emph{Step 5.}
In this step we introduce a boundary $\Delta$.
By the previous step, there exists $\alpha\in (0,1)$ depending only on $d,\epsilon$ such that 
$$
\Delta:=\alpha B_V+(1-\alpha)\Lambda\ge 0.
$$ 
Then, by Lemma \ref{l-average-boundary}, $(V,\Delta)$ is $\epsilon'$-lc where 
$\epsilon'=\alpha \epsilon$ because $(V,B_V)$ is sub-$\epsilon$-lc and $(V,\Lambda)$ is lc. 
Moreover, by Step 3,
$$
a(T,V,\Delta)=\alpha a(T,V,B_V)+(1-\alpha) a(T,V,\Lambda)
$$
$$
\le \alpha +(1-\alpha)=1.
$$

On the other hand, there is a bounded number $l\in\N$ such that $lH- \Lambda$ is ample. Moreover,  
by construction, $-B_V\sim_\R K_V$, hence we can assume  $lH-B_V\sim_\R lH+K_V$ is ample as well. 
This in turn implies 
$$
lH-\Delta=\alpha (lH-B_V)+(1-\alpha)(lH-\Lambda)
$$ 
is ample too. In addition, there is a bounded natural number $r$ 
such that $(lH)^d\le r$.\\

\emph{Step 6.}
In this step we finish the proof by applying Theorem \ref{t-bnd-lct}.
Let $M=\psi_*\phi^*uT$.
Since $\Omega\equiv -K_X \equiv uT$, the degree $\deg_HM=\deg_H(\psi_*\phi^*\Omega)$ is bounded from above, 
by Step 2, which implies the coefficients of $M$ are bounded from above. 
In particular, we may assume $T$ is exceptional over $V$, otherwise $u$ would be bounded. 
Thus $M$ is exceptional over $X$, hence its support is inside $\Lambda$. So perhaps after replacing 
$l$ we can assume $lH-M$ is ample.

On the other hand, since $T$ is ample, $\phi^*uT\le \psi^*M$, by the negativity lemma, hence 
the coefficient of the birational transform of $T$ in $\psi^*M$ is at least $u$.  
Therefore, $(V,\Delta+\frac{1}{u}M)$ is not klt as 
$$
a(T,V,\Delta+\frac{1}{u}M)\le a(T,V,\Delta)-1\le 0.
$$
Now by Theorem \ref{t-bnd-lct}, there is a positive number $t$ depending only on  
$d,\epsilon',r$ such that $(V,\Delta+tM)$ is klt. 
Therefore,  $t<\frac{1}{u}$, hence $u<v:=\frac{1}{t}$.

\end{proof}

\begin{lem}\label{l-local-lct-bab-to-global-lct}
Assume  that Theorem \ref{t-bnd-lct} holds in dimension $\le d$ and that Theorem \ref{t-BAB} 
holds in dimension $\le d-1$. Then Theorem \ref{t-bnd-lct-global}
holds in dimension $d$. 
\end{lem}
\begin{proof}
Pick $\epsilon'\in (0,\epsilon)$.
Let $(X,B)$ and $A=-(K_X+B)$ be as in Theorem \ref{t-bnd-lct-global} in dimension $d$ and pick $L\in|A|_\R$. 
Let $s$ be the largest number such that $(X,B+sL)$ is $\epsilon'$-lc.
It is enough to show $s$ is bounded from below away from zero. In particular, we may assume $s<1$.
Replacing $X$ with a $\Q$-factorialisation, we can assume $X$ is $\Q$-factorial. 
There is a prime divisor $T$ over $X$, that is, on birational models of $X$, with log discrepancy 
$$
a(T,X,B+sL)=\epsilon'.
$$
If $T$ is not exceptional over $X$, then  we let $\phi\colon Y\to X$ be the identity morphism. 
But if $T$ is exceptional over $X$, 
then we let $\phi\colon Y\to X$ be the extremal birational contraction which extracts $T$. 
Let $K_{Y}+B_Y=\phi^*(K_{X}+B)$ and let $L_Y=\phi^*L$. By assumption, 
$$
\mu_{T}B_Y\le 1-\epsilon ~~~\mbox{but}~~~ \mu_{T}(B_Y+sL_Y)=1-\epsilon',
$$
 hence 
$\mu_{T}sL_Y\ge \epsilon-\epsilon'$. 

Since $s<1$, 
$$
-(K_X+B+sL)= A-sL\sim_\R (1-s)A
$$ 
is nef and big, hence $-(K_Y+B_Y+sL_Y)$ is also nef and big. 
Thus $(Y,B_Y+sL_Y)$ is klt weak log Fano, so $Y$ is of Fano type. Run an MMP on $-T$ and let $Y'\to Z'$ 
be the resulting Mori fibre space. Then 
$$
-(K_{Y'}+B_{Y'}+sL_{Y'})\sim_\R (1-s)L_{Y'}\ge 0.
$$
Moreover, $({Y'},B_{Y'}+sL_{Y'})$ is $\epsilon'$-lc because $(Y,B_Y+sL_Y)$ is $\epsilon'$-lc and 
$-(K_{Y}+B_{Y}+sL_{Y})$ is semi-ample.
 If $\dim Z'>0$, then restricting to a general fibre of $Y'\to Z'$ and applying 
Theorem \ref{t-bnd-lct-global} in lower dimension by induction (or applying Theorem \ref{t-BAB}) 
shows that the coefficients of the horizontal$/Z'$ components 
of $(1-s)L_{Y'}$ are bounded from above. In particular, $\mu_{T'}(1-s)L_{Y'}$ is bounded from above. 
Thus from the inequality 
$$
\mu_{T'}(1-s)L_{Y'}\ge \frac{(1-s)(\epsilon-\epsilon')}{s},
$$ 
we deduce that $s$ is bounded from below away from zero. 
Therefore, we can assume $Z'$ is a point and that $Y'$ is a Fano variety with Picard number one. 
Now 
$$
-K_{Y'}\sim_\R L_{Y'}+B_{Y'}=(1-s)L_{Y'}+sL_{Y'}+B_{Y'}\ge (1-s)L_{Y'},
$$ 
so by Proposition \ref{p-bnd-lct-global-weak}, $\mu_{T'}(1-s)L_{Y'}$ is bounded from above 
which again gives a lower bound for $s$ as before.

\end{proof}

Next we treat Theorem \ref{t-global-lct-attained}. As mentioned in the introduction, the theorem and 
its proof are independent of the rest of this paper when $\lct(X,B,|A|_\R)<1$.

\begin{lem}\label{l-t-global-lct-attained=1}
Assume Theorem \ref{t-bnd-lct} holds in dimension $d$.
Then Theorem \ref{t-global-lct-attained} holds in dimension $d$ when $\lct(X,B,|A|_\R)=1$.
\end{lem}
\begin{proof}
By definition of the lc threshold, there exists a sequence of $\R$-divisors 
$$
0\le L_i\sim_\R A=-(K_X+B)
$$ 
such that the numbers $t_i:=\lct(X,B,L_i)$ form a decreasing sequence with $\lim t_i=1$. 

Assume $(X,B)$ is not exceptional, that is, assume there is $0\le L\sim_\R A$ such that $(X,B+L)$ is not 
klt. Thus $\lct(X,B,L)\le 1$. Since  $\lct(X,B,|A|_\R)=1$, we deduce $\lct(X,B,L)=1$ and that 
$(X,B+L)$ is lc but not klt. If $B$ is a $\Q$-boundary, 
then using approximation, we can replace $L$ so that $0\le L\sim_\Q A$, so we are done in this case by 
taking $D=L$.

Now assume $(X,B)$ is exceptional,  that is, for any $0\le L\sim_\R A$ the pair $(X,B+L)$ is klt. 
We will derive a contradiction. 
By [\ref{B-compl}, Lemma 7.2], there is $\epsilon>0$ such that $(X,B+L_i)$ 
is $\epsilon$-lc for every $i$. Then by Theorem \ref{t-bnd-lct} in dimension $d$, there is 
$s>0$ such that $(X,B+L_i+sL_i)$ is klt, for every $i$. But then $1+s<t_i$ for every $i$ which contradicts 
$\lim t_i=1$.

\end{proof}

\begin{prop}\label{p-t-global-lct-attained<1}
Theorem \ref{t-global-lct-attained} holds when $\lct(X,B,|A|_\R)<1$.
\end{prop}
\begin{proof}
\emph{Step 1.}
From here to the end of Step 4 we find $D\in |A|_\R$
such that 
$$
\lct(X,B,|A|_\R)=\lct(X,B,D).
$$
In Step 5 we treat the case when $B$ is a $\Q$-boundary.
 In this step we make some preparations.
  
Replacing $X$ with a $\Q$-factorialisation, we can assume $X$ is $\Q$-factorial. 
By definition of the lc threshold, there exists a sequence of $\R$-divisors 
$$
0\le L_i\sim_\R A=-(K_X+B)
$$ 
such that the numbers $t_i:=\lct(X,B,L_i)\in(0,1)$ form a decreasing sequence with
$$
t:=\lct(X,B,|A|_\R)=\lim t_i.
$$
If $t=t_i$ for some $i$, then put $D=L_i$. 
So for now assume $t\neq t_i$ for every $i$.\\ 

\emph{Step 2.}
In this step we choose elements $H_i\in |A|_\R$ and define birational contractions $X_i'\to X$ 
with respect to some lc places of $(X,B+t_iL_i)$.
Pick $H_i\in |A|_\R$  so that   
$(X,B+H_i)$ is klt and $(X,B+t_iL_i+H_i)$ is lc, for any $i$, and such that the coefficients of 
the $H_i$ belong to a fixed DCC set. Existence of the DCC set follows from the fact that $A$ is semi-ample: 
indeed, $A\sim_\R \sum \alpha_j G_j$ 
where $G_j$ are base point free Cartier divisors and $\alpha_j>0$; so to find the $H_i$ we only need 
to move the $G_j$ appropriately. 
Note that 
$$
K_X+B+t_iL_i+(1-t_i)H_i\sim_\R 0.
$$

Let $T_i'$ be an lc place of $(X,B+t_iL_i)$. If $T_i'$ is not exceptional over $X$, we let 
$\phi_i\colon X_i'\to X$ to be the identity morphism, but if $T_i'$ is exceptional over $X$, 
we let $\phi_i\colon X_i'\to X$ be the extremal birational contraction 
which extracts $T_i'$. Let $B_i',L_i', H_i'$ be the birational transforms of $B,L_i,H_i$. 
Then $H_i'=\phi_i^*H_i$ because $H_i$ cannot contain any lc centre of $(X,B+t_iL_i)$. Moreover, 
$$
K_{X_i'}+T_i'+B_i'+t_iL_i'+(1-t_i)H_i'=\phi_i^*(K_X+B+t_iL_i+(1-t_i)H_i)\sim_\R 0. 
$$\ 

\emph{Step 3.}
In this step we run an MMP on 
$$
-(K_{X_i'}+T_i'+B_i'+(1-t)H_i')
$$ 
and study the outcome. Note that here we have used $t$ rather than $t_i$. 
By [\ref{B-compl}, 2.13(7)], $X_i'$ is of Fano type, so we can indeed run such an MMP.
Let $X_i''$ be the resulting model of the MMP. 
Then 
$$
({X_i''},T_i''+B_i''+t_iL_i''+(1-t_i)H_i'')
$$
is lc, and the coefficients of $T_i''+B_i''$ and $H_i''$ belong to a fixed DCC set independent of $i$. 
Since the numbers $1-t_i$ form an increasing sequence approaching $1-t$, 
 by the ACC for lc thresholds [\ref{HMX2}], we can assume 
$$
({X_i''},T_i''+B_i''+(1-t)H_i'')
$$
is lc, for every $i$. 

Now assume the MMP ends with a Mori fibre space $X_i''\to Z_i''$, for infinitely many $i$. Then 
for such $i$,
$$
K_{X_i''}+T_i''+B_i''+(1-t)H_i''
$$
is ample over $Z_i''$. On the other hand, 
$$
K_{X_i''}+T_i''+B_i''+t_iL_i''+(1-t_i)H_i''\sim_\R 0
$$
which implies 
$$
K_{X_i''}+T_i''+B_i''+(1-t_i)H_i''
$$
is anti-nef over $Z_i''$.  This contradicts  [\ref{HMX2}, Theorem 1.5] by restricting to 
the general fibres of $X_i''\to Z_i''$. Thus replacing the sequence with a subsequence, we can assume the MMP 
ends with a minimal model $X_i''$, for every $i$.\\

\emph{Step 4.}
In this step we find the desired $D\in |A|_\R$ such that 
$
\lct({X},B,D)=t.
$  
Fix $i$.
Since 
$$
-(K_{X_i''}+T_i''+B_i''+(1-t)H_i'')
$$
is nef, hence semi-ample, it is $\R$-linearly equivalent to some $\R$-divisor $P_i''\ge 0$. 
Since $X_i'\bir X_i''$ is an MMP on 
 $$
-(K_{X_i'}+T_i'+B_i'+(1-t)H_i'),
$$ 
we get an $\R$-divisor $P_i'\ge 0$ such that 
$$
-(K_{X_i'}+T_i'+B_i'+(1-t)H_i')\sim_\R P_i'
$$
which in turn gives an $\R$-divisor $P_i\ge 0$ such that 
$$
-(K_{X}+B+(1-t)H_i)\sim_\R P_i. 
$$
By construction, 
$$
({X},B+(1-t)H_i+P_i)
$$
is not klt near the generic point of the centre of $T_i'$ on $X$. 
Moreover, $P_i\sim_\R tA$, and $(X,B+P_i)$ is not klt as $\Supp H_i$ does not contain the centre of $T_i$. 
Put $D=\frac{1}{t}P_i$. Then $D\sim_\R A$ and $\lct(X,B,D)\le t$, so the inequality is actually an equality 
by definition of $t$.\\

\emph{Step 5.}
In this step we treat the case 
when $B$ is a $\Q$-boundary. Fix $i$. Since $A_i$ is a $\Q$-divisor, we can assume $H_i$ is a $\Q$-divisor.
If $t$ is a rational number, then in Step 4 we can take $P_i''$ to be a $\Q$-divisor, 
hence $P_i$ is also a $\Q$-divisor which 
in turn means we can choose $D$ to be a $\Q$-divisor. Assume now that $t$ is not a rational number. 
We will derive a contradiction. Assume 
$$
-(K_{X_i''}+T_i''+B_i''+(1-t)H_i'')
$$
is not big. Then, since it is semi-ample, it is numerically trivial on some covering family of curves, 
hence taking intersection 
with such curves $C$ and using the fact that $H_i''\cdot C\neq 0$ (as $H_i''$ is big) ensures that $t$ is a rational number, a contradiction. 
On the other hand, if the above divisor is big, then 
$$
-(K_{X_i'}+T_i'+B_i'+(1-t)H_i')
$$ 
is big and this in turn implies 
$$
-(K_{X_i'}+T_i'+B_i'+(1-e)H_i')
$$
is also big for some rational number $e\in (0,t)$. 
Thus as in Step 4 we can find $0\le Q_i\sim_\Q eA$ such that 
$(X,B+Q_i)$ is not klt. Letting $D=\frac{1}{e}Q_i$ gives $0\le D\sim_\Q A$ with 
$$
\lct(X,B,D)\le e<t,
$$ 
contradicting the definition of $t$.

\end{proof}


\section{\bf Complements in a neighbourhood of a non-klt centre}

In this section, we prove our main result on the existence of complements (Theorem \ref{t-compl-near-lcc}). 
It does not follow directly from [\ref{B-compl}] but the proofs in [\ref{B-compl}] 
work with appropriate modifications. First we treat a special case of the theorem.

\begin{prop}\label{l-bnd-comp-Gamma}
Theorem \ref{t-compl-near-lcc} holds under the additional assumption that there is a boundary $\Gamma$ such that
\begin{itemize}
\item  $(X,\Gamma)$ is plt with $S=\rddown{\Gamma}$, and 
 
\item $\alpha M-(K_X+\Gamma)$ is ample for some real number $\alpha>0$. 
\end{itemize}
\end{prop}
\begin{proof}
 
\emph{Step 1.} 
{In this step we reduce to the case when $\alpha\in (0,2)$ and when $B-\Gamma$ has 
small (positive or negative) coefficients.}
Since $M-(K_X+B)$ is nef and big and  $\alpha M-(K_X+\Gamma)$ is ample,
$$
(1-t+t\alpha)M-(K_X+(1-t)B+t\Gamma)=(1-t+t\alpha)M-(1-t)(K_X+B)-t(K_X+\Gamma)
$$
$$
=(1-t)(M-(K_X+B))+t(\alpha M-(K_X+\Gamma))
$$
is ample for any $t\in (0,1)$. Thus replacing $\Gamma$ 
with $(1-t)B+t\Gamma$ for some sufficiently small real number $t>0$, 
we can replace $\alpha$ by some rational number in $(0,2)$. Note that since $(X,B)$ is lc and 
since $S$ is a non-klt centre of this pair, we have $S\le \rddown{B}$ and 
the above change of $\Gamma$ preserves the plt property of $(X,\Gamma)$ and the condition $S=\rddown{\Gamma}$.

It is also clear that if we choose $t$ small enough then we can ensure that 
$B-\Gamma$ has sufficiently small positive or negative coefficients (we will use this in steps below).\\

\emph{Step 2.} 
{In this step we consider bounded complements on $S$.}
Since $(X,\Gamma)$ is plt, $S$ is normal. Define 
$$
K_S+B_S=(K_X+B)|_S
$$ 
by adjunction. Then $(S,B_S)$ is lc and 
the coefficients of $B_S$ belong to $\Phi(\mathfrak{S})$ 
for some finite set $\mathfrak{S}\subset [0,1]$ of rational numbers 
depending only on $p$ [\ref{PSh-II}, Proposition 3.8][\ref{B-compl}, Lemma 3.3] where 
$$
\Phi(\mathfrak{S})=\{1-\frac{s}{l} \mid s\in \mathfrak{S}, l\in \N\}.
$$

By assumption, $M|_S\equiv 0$, hence $M|_S\sim_\Q 0$ as $M$ is semi-ample. 
In particular, 
$$
-(K_X+\Gamma)|_S\sim_\R (\alpha M-(K_X+\Gamma))|_S
$$ 
is ample. Moreover, since $(X,\Gamma)$ is plt, defining 
$$
K_S+\Gamma_S=(K_X+\Gamma)|_S
$$ 
by adjunction, $(S,\Gamma_S)$ is klt. Thus we deduce that $S$ is of Fano type. 
On the other hand, 
$$
-(K_S+B_S)\sim_\Q (M-(K_X+B))|_S
$$ 
is nef. 
Therefore, applying Theorem \ref{t-bnd-compl}, there is a natural number $n$ depending only on $d,\mathfrak{S}$ 
such that $K_S+B_S$ has an $n$-complement $K_S+B_S^+$ which satisfies 
$B_S^+\ge B_S$. We can assume that $nB$ is an integral divisor after replacing $n$ with $np$.

Note that since $S$ is of Fano type and $M|_S\sim_\Q 0$, 
in fact we have $M|_S\sim 0$ as $\Pic(S)$ is torsion-free (cf. [\ref{Isk-Prokh}, Proposition 2.1.2]; the  
point is that $h^i(\mathcal{O}_S)=0$ for $i>0$ which is a consequence of 
Kawamata-Viehweg vanishing theorem).\\

\emph{Step 3}.
{In this step we take a resolution of $X$ and define appropriate divisors on it.}
Let $\phi\colon X'\to X$ be a log resolution of $(X,B+\Gamma)$, $S'\subset X'$ be the birational transform of $S$, and 
$\psi\colon S'\to S$ be the induced morphism. Put
$$
N:=M-(K_{X}+B)
$$
and let $K_{X'}+B',M',N'$ be the pullbacks of 
$K_X+B,M,N$, respectively. Let $E'$ be the sum of the components of $B'$ which have coefficient $1$, 
and let $\Delta'=B'-E'$. Define  
$$
L':=(n+2)M'-nK_{X'}-nE'-\rddown{(n+1){\Delta'}}
$$
which is an integral divisor. Note that 
$$
L'=(n+2)M'-nK_{X'}-nB'+n\Delta'-\rddown{(n+1){\Delta'}}
$$
$$
=2M'+n(M'-K_{X'}-B')+n\Delta'-\rddown{(n+1){\Delta'}}
$$
$$
=2M'+nN'+n\Delta'-\rddown{(n+1){\Delta'}}.
$$
Now write 
$$
K_{X'}+\Gamma'=\phi^*(K_{X}+\Gamma).
$$ 
Then we can assume that $B'-\Gamma'=\phi^*(B-\Gamma)$ has sufficiently small (positive or negative) coefficients by
Step 1. The coefficient of $S'$ in both $B'$ and $\Gamma'$ is $1$ but its coefficient in $\Delta'$ is $0$.\\ 

\emph{Step 4}.
{In this step we introduce a boundary $\Theta'$ and study related divisors.}
Let $P'$ be the unique integral divisor so that 
$$
\Theta':=\Gamma'+{n{\Delta'}}-\rddown{(n+1){\Delta'}}+P'
$$ 
is a boundary, $(X',\Theta')$ is plt, and $\rddown{\Theta'}=S'$ (in particular, we are assuming $\Theta'\ge 0$). 
More precisely, we let $\mu_{S'}P'=0$ and for each prime divisor $D'\neq S'$, we let 
$$
\mu_{D'}P':=-\mu_{D'}\rddown{\Gamma'+{n{\Delta'}}-\rddown{(n+1){\Delta'}}}
$$
which satisfies 
$$
\mu_{D'}P'=-\mu_{D'}\rddown{\Gamma'-\Delta'+{(n+1){\Delta'}}-\rddown{(n+1){\Delta'}}}
$$
$$
=-\mu_{D'}\rddown{\Gamma'-\Delta'+\langle (n+1)\Delta'\rangle}
$$ 
where $\langle (n+1)\Delta'\rangle$ is the fractional part of $(n+1)\Delta'$.

We claim that $P'\ge 0$. Pick a component $D'$ of $P'$. By the above, $D'\neq S'$. 
If $D'$ is a component of $E'$, then 
$D'$ is not a component of $\Delta'$ and $\mu_{D'}\Gamma'\in(0,1)$ as $B'-\Gamma'$ has 
small coefficients and $\mu_{D'}B'=1$, hence $\mu_{D'}P'=0$; 
on the other hand, if $D'$ is not a 
component of $E'$, then the absolute value of 
$$
\mu_{D'}(\Gamma'-\Delta')=\mu_{D'}(\Gamma'-B')
$$ 
is sufficiently small and 
$$
\mu_{S'}\langle (n+1)\Delta'\rangle\in [0,1),
$$ 
hence $\mu_{D'}P'=0$ or $\mu_{D'}P'=1$, so in any case $\mu_{D'}P'\ge 0$. 

We show $P'$ is exceptional$/X$. 
Assume $D'$ is a component of $P'$ which is not exceptional$/X$ and let $D$ be its pushdown. 
Since $S'$ is not a component of $P'$, $D'\neq S'$.
Let $E,\Delta$ be the pushdowns of $E',\Delta'$. 
Since $nB$ and $nE$ are integral, $\mu_{D}n\Delta$ is integral, hence  
$$
\mu_{D}\rddown{(n+1)\Delta}=\mu_{D}n\Delta.
$$
Then 
$$
\mu_{D'}P'=-\mu_{D'}\rddown{\Gamma'+{n{\Delta'}}-\rddown{(n+1){\Delta'}}}
$$
$$
=-\mu_{D}\rddown{\Gamma+{n{\Delta}}-\rddown{(n+1){\Delta}}}
=-\mu_{D}\rddown{\Gamma}=0
$$ 
because $\mu_{D}\Gamma\in [0,1)$, a contradiction.\\

\emph{Step 5.}
{In this step we show that sections of $(L'+P')|_{S'}$ can be lifted to $X'$.}
Let 
$$
A:=\alpha M-(K_X+\Gamma).
$$ 
Letting  $A'=\phi^*A$ we have 
$$
K_{X'}+\Gamma'+A'-\alpha M'=0.
$$ 
Then 
$$
L'+P'= 2M'+nN'+ {n{\Delta'}}-\rddown{(n+1){\Delta'}}+ P'
$$
$$
=K_{X'}+\Gamma'+A'-\alpha M'+2M'+nN'+{n{\Delta'}}-\rddown{(n+1){\Delta'}}+ P'
$$
$$
=K_{X'}+\Theta'+A'+nN'+(2-\alpha) M' . 
$$
By Step 1, $2-\alpha\ge 0$, so 
$$
A'+nN'+(2-\alpha) M'
$$ 
is nef and big. Moreover, $(X',\Theta')$ is plt with $\rddown{\Theta'}=S'$, so  
we have 
$$
h^1(L'+P'-S')=0
$$ 
by the Kawamata-Viehweg vanishing theorem as $(X',\Theta'-S')$ is klt. Thus the restriction map
$$
H^0(L'+P')\to H^0((L'+P')|_{S'})
$$
is surjective.\\

\emph{Step 6.}
{In this step we introduce an effective divisor $G_{S'}\sim (L'+P')|_{S'}$.}
Recall the $n$-complement $K_S+B_S^+$ from Step 2.
Let $R_{S}:=B_{S}^+-B_{S}$ which satisfies 
$$
-n(K_{S}+B_{S})=-n(K_{S}+B_{S}^++B_S-B_S^+)\sim -n(B_S-B_S^+)=nR_{S}\ge 0.
$$
Let $R_{S'}$ be the pullback of $R_{S}$. Since $M'$ is Cartier and since $S'$ is a 
component of $B'$ with coefficient $1$, the divisor 
$$
N'|_{S'}=(M'-(K_{X'}+B'))|_{S'}
$$
is well-defined up to linear equivalence. 
Now since $M'|_{S'}\sim 0$ by Step 2, we have   
$$
nN'|_{S'}=n(M'-(K_{X'}+B'))|_{S'}\sim-n(K_{X'}+B')|_{S'}
$$
$$
=-n\psi^*(K_{S}+B_{S})\sim n\psi^*R_S=nR_{S'}\ge 0. 
$$
Then 
$$
(L'+P')|_{S'}=( 2M'+ nN'+{n{\Delta'}}-\rddown{(n+1){\Delta'}}+P')|_{S'}
$$
$$
\sim G_{S'}:=nR_{S'}+{n{\Delta_{S'}}}-\rddown{(n+1){\Delta_{S'}}}+P_{S'}
$$
where $\Delta_{S'}=\Delta'|_{S'}$ and $P_{S'}=P'|_{S'}$. Note that 
$$
\rddown{(n+1){\Delta'}}|_{S'}=\rddown{(n+1){\Delta'}|_{S'}}
$$ 
since $\Delta'$ and $S'$ intersect transversally.

We show $G_{S'}\ge 0$. Assume $C'$ is a component of $G_{S'}$ 
with negative coefficient. Then since $R_{S'}$ and $P_{S'}$ are effective, 
 there is a component $D'$ of ${{\Delta'}}$ such that $C'$ is a 
component of $D'|_{S'}$. But 
$$
\mu_{C'} ({n{\Delta_{S'}}}-\rddown{(n+1){\Delta_{S'}}})=\mu_{C'} (-\Delta_{S'}+\langle(n+1)\Delta_{S'}\rangle)\ge 
-\mu_{C'} \Delta_{S'}=-\mu_{D'} \Delta'>-1
$$ 
which gives $\mu_{C'}G_{S'}>-1$ and this in turn implies $\mu_{C'}G_{S'}\ge 0$ because $G_{S'}$ is integral, a contradiction. 
Therefore $G_{S'}\ge 0$, and by Step 5, $L'+P'\sim G'$ for some effective divisor $G'$ whose support does not contain $S'$ 
and $G'|_{S'}=G_{S'}$.\\ 

\emph{Step 7.}
{In this step we introduce $\Lambda$ and show that it satisfies the properties listed in the theorem.}
Let $L,P,G$ be the pushdowns to $X$ of $L',P',G'$.
By the definition of $L'$, by the previous step, and by the exceptionality of $P'$,   we have  
$$
(n+2)M-nK_{X}-nE-\rddown{(n+1)\Delta}=L=L+P\sim G\ge 0.
$$
Since $nB$ is integral, $\rddown{(n+1)\Delta}= n\Delta$, so  
$$
(n+2)M-n(K_{X}+B)
$$
$$
=(n+2)M-nK_{X}-nE-{n\Delta}
=L\sim nR:=G\ge 0.
$$

Let $\Lambda:={B}^+:=B+R$. For consistency of notation, for the rest of the proof we will 
use $B^+$ instead of $\Lambda$. By construction, 
$$
n(K_X+B^+)\sim (n+2)M.
$$ 
It remains to show that $(X,{B}^+)$ is lc over $z=f(S)$. First we show that $(X,{B}^+)$ is lc near $S$:
this  follows from inversion of adjunction [\ref{kawakita}], if we show 
$$
K_{S}+B_{S}^+=(K_{X}+{B}^+)|_{S}
$$
which is equivalent to showing $R|_{S}=R_{S}$.
Since  
$$
nR':=G'-P'+\rddown{(n+1)\Delta'}- n\Delta'\sim L'+\rddown{(n+1)\Delta'}- n\Delta'=2M'+nN'\sim_\Q 0/X
$$
and since $\rddown{(n+1)\Delta}- n\Delta=0$, we get  $\phi_*nR'=G=nR$ and that $R'$ is the pullback of $R$. 
Now by Step 6,
$$
nR_{S'}=G_{S'}-P_{S'}+\rddown{(n+1)\Delta_{S'}}- n\Delta_{S'}
$$
$$
=(G'-P'+\rddown{(n+1)\Delta'}- n\Delta')|_{S'}=nR'|_{S'}
$$
which means $R_{S'}=R'|_{S'}$, hence $R_S$ and $R|_S$ both pull back to $R_{S'}$ which implies
$R_{S}=R|_{S}$ as required.

Assume now that $(X,{B}^+)$ is not lc over $z=f(S)$. By the previous paragraph, $(X,B^+)$ is lc near $S$ 
and $S$ is a non-klt centre of this pair. On the other hand, 
$(X,\Gamma)$ is plt with $\rddown{\Gamma}=S$, so if $u>0$ is sufficiently small, then  
$$
(X,(1-u)B^++u\Gamma)
$$
is plt near $S$ and $S$ is a non-klt centre of this pair and no other non-klt centre intersects $S$. 
Then since $(X,B^+)$ is not lc over $z$, 
the non-klt locus of  
$$
(X,(1-u)B^++u\Gamma)
$$
has at least two connected components (one of which is $S$) near the fibre $f^{-1}\{z\}$. This  
contradicts the connectedness principle [\ref{Kollar-flip-abundance}, Theorem 17.4]
as 
$$
-(K_X+(1-u)B^++u\Gamma)=-(1-u)(K_X+B^+)-u(K_X+\Gamma)
$$
$$
\sim_\R -u(K_X+\Gamma)\sim_\R u\alpha M-u(K_X+\Gamma)/Z
$$
is ample over $Z$. Therefore, $(X,B^+)$ is lc over $z$.

\end{proof}

\begin{proof}(of Theorem \ref{t-compl-near-lcc})
\emph{Step 1.}
{In this step we take a $\Q$-factorialisation of $X$ and consider the contraction defined 
by $aM-(K_X+B)$ for some $a>1$.}
Replacing $(X,B)$ with a $\Q$-factorial dlt model, 
we can assume $X$ is $\Q$-factorial and 
that $S$ is a component of $\rddown{B}$. All the assumptions of the theorem are preserved. 
The Fano type property is preserved by a relative version of [\ref{B-compl}, 2.13(7)] as $-(K_X+B)$ is nef over $Z$.

By assumption, $M$ is the pullback of an ample divisor on $Z$. Moreover, 
$M-(K_X+B)$ is nef and big, hence in particular it is semi-ample over $Z$ 
since $X$ is of Fano type over $Z$. Thus by Lemma \ref{l-semi-ample-div}, if $a>1$ is a real number, then 
$aM-(K_X+B)$ is semi-ample globally, so
it defines a birational contraction $X\to U$ which is simply the contraction over $Z$ defined 
by $M-(K_X+B)$: indeed, for any curve $C$ on $X$, 
$$
\mbox{$(aM-(K_X+B))\cdot C=0$ iff $(M-(K_X+B))\cdot C=0$ and $M\cdot C=0$.}
$$
In particular, $X\to U$ is birational as $aM-(K_X+B)$ is big, 
the induced map $U\bir Z$ is a morphism and $K_X+B\sim_\Q 0/U$.\\ 

\emph{Step 2.}
{In this step we make some further modifications of $(X,B)$ using the MMP.}
Since $X$ is of Fano type over $Z$, it is also of Fano type over $U$. 
 Run an MMP over $U$ on $-(K_X+S)$ and let $X'$ be the resulting model. The MMP does not contract $S$ 
as 
$$
B-S\sim_\Q -(K_X+S)/U
$$ 
and $S$ is not a component of $B-S$. Assume that there exist $n\in \N$ and $\Lambda'\ge B'$ 
such that $(X',\Lambda')$ is lc over $z$ and 
$$
n(K_{X'}+\Lambda')\sim (n+2)M'
$$ 
where 
$B',M'$ are the pushdowns of $B,M$. Then since $K_X+B\sim_\Q 0/U$, taking the crepant pullback 
of $K_{X'}+\Lambda'$ to $X$ we get $\Lambda\ge B$ such that $(X,\Lambda)$ is lc over $z$ and 
$$
n(K_{X}+\Lambda)\sim (n+2)M.
$$ 
Note that $M'$ is Cartier as the Cartier property of $M$ is preserved by the MMP, by the cone theorem.
Thus replacing $X$ with $X'$,  
we can assume $-(K_X+S)$ is semi-ample over $U$ defining a contraction $X\to V/U$. 
Moreover, every non-klt centre of $(X,S)$ is contained in $S$ because 
$(X,0)$ is klt as $X$ is $\Q$-factorial and of Fano type over $Z$. 

We claim that $S$ is not contracted over $V$: otherwise since $K_X+S\sim_\Q 0/V$, the pair $(V,0)$ would be lc 
but not klt and this is a contradiction because there is a boundary $\Theta$ such that 
$(V,\Theta)$ is klt as $V$ is of Fano type over $Z$. 

Note that the pushdown of $M$ to $V$ is Cartier, again by the cone theorem.
Replacing $X$ with $V$, we can then assume that $-(K_X+S)$ is ample over $U$. The $\Q$-factorial 
property of $X$ maybe lost but we do not need it any more.\\ 

\emph{Step 3.}
{In this step we introduce a boundary $\Delta$ and study some of its properties.}
Let 
$$
\Delta=(1-b)B+bS
$$ 
for a sufficiently small real number $b>0$ (depending on $a$). Then $(X,\Delta)$ is lc and $S$ is a 
non-klt centre of this pair. Moreover, every non-klt place of $(X,\Delta)$ is a non-klt place of both 
$(X,B)$ and $(X,S)$, in particular, every non-klt centre of $(X,\Delta)$ is 
also a non-klt centre of $(X,S)$, hence such centres are contained in $S$ and they are mapped to $z$. 

On the other hand,  
since $aM-(K_X+B)$ is the pullback of an ample divisor on $U$ and $aM-(K_X+S)$ is ample over $U$, 
we see that 
$$
aM-(K_X+\Delta)=aM-(K_X+(1-b)B+bS)
$$
$$
=aM-(1-b)(K_X+B)-b(K_X+S)
$$
$$
=(1-b)(aM-(K_X+B))+b(aM-(K_X+S))
$$ 
 is globally ample.\\

\emph{Step 4.}
{In this step we produce a plt pair and apply Proposition \ref{l-bnd-comp-Gamma} to finish the proof.}
By Lemma \ref{l-plt-blowup} applied to $(X,\Delta)$, there exist a prime divisor $T$ over $X$ and a projective birational 
morphism $Y\to X$ 
such that 
\begin{itemize}
\item either $Y\to X$ is small or it contracts $T$ but no other divisors,

\item $(Y,T)$ is plt,

\item $-(K_Y+T)$ is ample over $X$, and 

\item $a(T,X,\Delta)=0$.
\end{itemize}
In particular, $T$ is mapped to $z$, by Step 3, as it is an lc place of $(X,S)$. Moreover, since $\Delta\le B$, 
$$
a(T,X,B)=0.
$$

Let $K_Y+\Delta_Y$ be the pullback of $K_X+\Delta$.
Define 
$$
\Gamma_Y=(1-v) \Delta_Y+vT
$$ 
for some sufficiently  small $v>0$. 
Let $M_Y$ be the pullback of $M$. Let $\alpha=(1-v)a$. Since $aM-(K_X+\Delta)$
is ample and since $-(K_Y+T)$ is ample over $X$, 
$$
\alpha M_Y-(K_Y+\Gamma_Y)=(1-v)aM_Y-(K_Y+(1-v)\Delta_Y+vT)
$$
$$
=(1-v)a M_Y-(1-v)(K_Y+\Delta_Y)-v(K_Y+T)
$$
$$
=(1-v)(a M_Y-(K_Y+\Delta_Y))-v(K_Y+T)
$$
is ample. Moreover, $(Y,\Gamma_Y)$ is plt and $T=\rddown{\Gamma_Y}$ maps to $z$. 

Let $K_Y+B_Y$ be the pullback of $K_X+B$. By definition of $\Delta$, $T$ is an lc place of $(X,B)$, hence 
it appears in $B_Y$ with coefficient $1$.
Now we can replace $(X,B),M,S$ with $(Y,B_Y),M_Y,T$ and apply Proposition \ref{l-bnd-comp-Gamma}.

\end{proof}


\section{\bf Singularities of divisors with bounded degree}

In this section, we make necessary preparations for the proof of Theorem \ref{t-bnd-lct}. 
This involves a reduction to the case of projective space and eventually to the toric version 
of Theorem \ref{t-BAB} which is well-known [\ref{A-L-Borisov}].

\subsection{Finite morphisms to the projective space}\label{ss-etale}
We prove a version of Noether normalisation theorem. Part of it  
is similar to [\ref{Kedlaya}, Theorem 2] proved for fields of positive characteristic.

\begin{prop}\label{p-map-to-p^n}
Let $(X,\Lambda=\sum_1^d S_i)$ be a projective log smooth pair of dimension $d$ where 
$\Lambda$ is reduced. Let $B=\sum b_jB_j\ge 0$ be an $\R$-divisor.
Assume 
\begin{itemize}

\item $x\in  \bigcap_1^d S_i$,

\item $\Supp B$ contains no stratum of $(X,\Lambda)$ except possibly $x$, and 

\item $A$ is a very ample divisor such that $A-S_i$ is very ample for each $i$.\\
\end{itemize}
Then there is a finite morphism 
$$
\pi\colon X\to \PP^d=\Proj k[t_0,\dots,t_d]
$$ 
such that 
\begin{itemize}

\item $\pi(x)=z:=(1:0:\cdots:0)$, 

\item $\pi(S_i)=H_i$ where $H_i$ is the hyperplane defined by $t_i$, 
  
\item $\pi$ is \'etale over a neighbourhood of $z$, 

\item $\Supp B$ contains no point of $\pi^{-1}\{z\}$ except possibly $x$, and 

\item $\deg \pi={A^d}$ and $\deg_{H_i} C\le \deg_AB$  where $C=\sum b_j\pi(B_j)$.  
\end{itemize} 
\end{prop}

\begin{proof}
Since $A-S_i$ is very ample for each $i$, taking general divisors $D_i\in |A-S_i|$ 
we can make sure $(X,\sum_1^d R_i)$ is log smooth where $R_i:=D_i+S_i\sim A$. 
Moreover, we can assume that $\Supp B$ contains no 
stratum of $(X,\sum_1^d R_i)$ other than $x$: since $D_1$ is general, it is not a component of 
$B$, and if $I\neq x$ is a stratum of $(X,\Lambda)$, then $D_1|_I$ has no common component
with $B|_I$; this ensures $\Supp B$ does not contain any stratum of $(X,\Lambda+D_1)$ 
other than $x$; repeating this process proves the claim. 

Now each $R_i$ is the zero divisor of  some global 
section $\alpha_i$ of $\mathcal{O}_X(A)$. Choose another global section $\alpha_{0}$ so that if 
$R_{0}$ is the zero divisor of $\alpha_{0}$, then $(X,\sum_0^{d} R_i)$ is still log smooth, and 
that  $\bigcap_0^{d} R_i$ is empty. 
The sections $\alpha_0,\dots,\alpha_{d}$ have no common vanishing point,  
so they define a morphism $\pi \colon X\to \PP^d$ so that $\mathcal{O}_X(A)\simeq \pi^*\mathcal{O}_{\PP^d}(1)$ and 
the global section $t_i$ of $\mathcal{O}_{\PP^d}(1)$ 
pulls back to $\alpha_i$, for each $0\le i \le d$. In particular, since $A$ is ample, $\pi$ does not contract any curve, 
hence $\pi$ is a surjective finite morphism. 

Since $t_i$ pulls back to $\alpha_i$, the zero divisor of $t_i$ pulls back to the zero divisor of $\alpha_i$, that is, 
$\pi^*H_i=R_i$. Thus $\pi(R_i)=H_i$ which in turn gives $\pi(S_i)=H_i$.  
Moreover, since 
$$
z:=(1:0:\cdots:0)=\bigcap_1^d H_i,
$$ 
we get 
$$
\pi^{-1}\{z\}=\bigcap_1^d\pi^{-1}H_i=\bigcap_1^dR_i
$$ 
which shows $\pi(x)=z$ as $x\in \bigcap_1^dS_i\subseteq \bigcap_1^dR_i$. 

The rest of the proof is elementary which uses \ref{l-etale} among other things and we leave it to the reader.

\end{proof}

\subsection{Bound on the length of blowup sequences}

We state a baby version of Theorem \ref{t-bnd-lct}, due to Viehweg (see [\ref{Viehweg}, Corollary 5.11]), 
before moving on to the main result of this subsection.

\begin{lem}\label{l-bnd-lct-no-boundary}
Let $d,r$ be natural numbers. Then there is a 
positive real number $t$ depending only on $d,r$ satisfying the following. 
Assume 
\begin{itemize}
\item $X$ is a smooth projective variety of dimension $d$,

\item $A$ is a very ample divisor on $X$ with $A^d\le r$, 

\item $L\ge 0$ is an $\R$-divisor on $X$ with $\deg_AL\le r$.\\
\end{itemize}
Then $(X,tL)$ is klt. 
\end{lem}
\begin{proof}
The assumptions imply that the multiplicity $\mu_xL$ is bounded from above, for each closed point $x\in X$. Thus 
taking $t$ small enough we have $\mu_xtL<1$. Now applying [\ref{Lazar}, Proposition 9.5.13] we deduce that 
$(X,tL)$ is klt. It is also easy to prove the lemma using induction on dimension and inversion of adjunction.

\end{proof}

It is not hard to extend the lemma and prove Theorem \ref{t-bnd-lct} when $(X,\Supp B)$ belongs to some 
bounded family of pairs [\ref{B-compl}, Proposition 4.2]. The general case of the theorem however
requires a lot more work and the main reason is that we have no control over the support of $B$. 
Much of the difficulties already appear in the case $X=\PP^d$. 

In the next key result, we bound the number of blowups in the centre blowup sequence associated to 
certain lc places.

\begin{prop}\label{p-bnd-lct-toroidal-1}
Let $d,r$ be natural numbers and $\epsilon$ be a positive real number. 
Then  there is a natural number $p$ depending only on $d,r,\epsilon$ satisfying the following. 
Assume 
\begin{itemize}

\item $(X,B)$ is a projective $\epsilon$-lc pair of dimension $d$, 

\item $A$ is a very ample divisor on $X$ with $A^d\le r$, 

\item $(X,\Lambda)$ is log smooth where $\Lambda\ge 0$ is reduced, 

\item $\deg_A B\le r$ and $\deg_A \Lambda\le r$, 

\item  $x$ is a zero-dimensional stratum of $(X,\Lambda)$,

\item  $\Supp B$ does not contain any stratum of $(X,\Lambda)$ except possibly $x$,

\item $T$ is an lc place of $(X,\Lambda)$ with centre $x$, and 

\item  $a(T,X,B)\le 1$.\\
\end{itemize}
Then $T$ can be obtained by a sequence of centre blowups,  toroidal with respect to $(X,\Lambda)$, 
of length at most $p$.
\end{prop}
\begin{proof}
We will prove the proposition over $\C$. Over other fields we can either apply the Lefschetz principle 
or simply use formal neighbourhoods instead of analytic neighbourhoods in the arguments below.\\ 

\emph{Step 1.} 
From here until the end of Step 3, we will relate the problem to a similar problem 
on the projective space $\PP^d$. In this step we consider a finite morphism to $\PP^d$. 
Removing the components of $\Lambda$ not passing through $x$, we can assume 
$\Lambda=\sum_1^d S_i$ where $S_i$ are the irreducible components. 
Since $\deg_A \Lambda\le r$, $(X,\Lambda)$ belongs to a bounded family of pairs depending only on 
$d,r$. Thus replacing $A$ with a bounded multiple and replacing $r$ accordingly, 
we can assume $A-S_i$ is very ample for each $i$. Writing $B=\sum b_jB_j$ where $B_j$ are the distinct 
irreducible components,  by Proposition \ref{p-map-to-p^n}, there is a finite morphism 
$$
\pi\colon X\to \PP^d=\Proj \C[t_0,\dots,t_d]
$$ 
 mapping $x$ to the origin $z=(1:0:\cdots:0)$ and mapping $S_i$ onto the hyperplane $H_i$ 
defined by $t_i$. Moreover, $\pi$ is \'etale over $z$, $\Supp B$ contains no point of 
$\pi^{-1}\{z\}$ other than $x$, $\deg \pi={A^d}$, and $\deg_{H_i} C\le \deg_AB\le r$  where 
$C=\sum b_j\pi(B_j)$. In addition, by the proof of \ref{p-map-to-p^n}, $\pi^*H_i$ coincides with $S_i$ 
near $x$.\\

\emph{Step 2.}
In this step we consider the centre blowup sequence of $T$ and the induced sequence associated to $\PP^d$. 
By Lemma \ref{l-etale-0}, 
there exist analytic neighbourhoods $U$ and $V$ of $x$ and $z$, respectively, such that 
$\pi^{\rm an}|_U$ induces an analytic isomorphism between the analytic pairs $(U,B^{\rm an}|_U)$ and $(V,C^{\rm an}|_V)$. 
In particular, $(\PP^d,C)$ is $\epsilon$-lc near $z$.
Let $\Theta:=\sum_1^d H_i$. Since $\pi^*H_i$ coincides with $S_i$ 
near $x$, we also have an analytic isomorphism between 
$(U,\Lambda^{\rm an}|_U)$ and $(V,\Theta^{\rm an}|_V)$. Moreover, 
each stratum of $(\PP^d,\Theta)$ passes through $z$ and each one is the   
image of a stratum of $(X,\Lambda)$ passing through $x$.
Thus $\Supp C$ does not contain any stratum of $(\PP^d,\Theta)$, except possibly $z$: indeed, if 
$\Supp C$ contains a stratum $I$, then there is a stratum $J$ of $(X,\Lambda)$ passing through $x$ 
which maps onto $I$; by the above analytic isomorphisms, $J^{\rm an}|_U\subseteq \Supp B^{\rm an}|_U$ 
which implies $J\subseteq \Supp B$, hence $J=x$ and $I=z$.

The centre of $T$ on $X$, that is $x$, is an lc centre and a stratum of $(X,\Lambda)$. 
Let $X_1\to X_0=X$ be the blowup of $X$ along this centre. Let $K_{X_1}+\Lambda_1$ be 
the pullback of $K_X+\Lambda$. Then $(X_1,\Lambda_1)$ is log smooth with $\Lambda_1$ reduced 
and containing the exceptional divisor of the blowup. Moreover, 
the centre of $T$ on $X_1$ is an lc centre of $(X_1,\Lambda_1)$, 
hence a stratum of $(X_1,\Lambda_1)$. We blowup $X_1$ along the centre of $T$ and so on. 
Thus we get a sequence  
$$
Y=X_l\to \cdots \to X_0=X
$$ 
of centre blowups obtaining $T$ as the exceptional divisor 
of the last blowup (\ref{ss-blowup-seq} (3)). The sequence is toroidal with respect to $(X,\Lambda)$. 

 Since the above sequence starts with blowing up $x$, and since 
$(U,\Lambda^{\rm an}|_U)$ and $(V,\Theta^{\rm an}|_V)$ are analytically isomorphic, the sequence 
corresponds to a sequence of blowups 
$$
W=Z_l\to \cdots \to Z_0=\PP^d
$$ 
which is the sequence of centre blowups of $R$, the exceptional divisor of $Z_l\to Z_{l-1}$. 
The latter sequence starts with blowing up $z$ and it is 
toroidal with respect to $(\PP^d,\Theta)$, hence $a(R,\PP^d,\Theta)=0$. On the other hand, 
since $(U,B^{\rm an}|_U)$ and $(V,C^{\rm an}|_V)$ are analytically isomorphic, we get 
$$
a(R,\PP^d,C)=a(T,X,B)\le 1.
$$ 

Note that we cannot simply replace $X,B,\Lambda$ with $\PP^d, C,\Theta$ because 
we do not know whether $(\PP^d,C)$ is $\epsilon$-lc away from $z$.\\ 

\emph{Step 3.}
In this step we show that there is a positive real number $t\in (0,\frac{1}{2})$ depending only on $d,r$ 
such that $(\PP^d,\Theta+tC)$ is lc away from $z$. Pick a closed point $y\in \PP^d$ other than $z$.
If $y$ is not contained in $\Theta$, we can apply Lemma \ref{l-bnd-lct-no-boundary}
to find $t>0$ bounded from below away from zero so that $(\PP^d,tC)$ is klt, hence $(\PP^d,\Theta+tC)$ is klt near $y$. 
Now assume $y$ is contained in $\Theta$, hence it is contained in some stratum $G$ of $(\PP^d,\Theta)$ 
of minimal dimension.  Note that $G$ is positive-dimensional and $\deg_{H'}C|_G=\deg_HC\le r$ 
where $H'$ on $G$ is the restriction of a general hyperplane $H$. Moreover, $G$ is not inside $\Supp C$, by Step 2. 
 Applying Lemma \ref{l-bnd-lct-no-boundary} again, we find $t>0$ bounded from below away from zero 
such that $(G,tC|_G)$ is klt.
Thus by inversion of adjunction,  $(\PP^d,\Theta+tC)$ is lc near $y$ because in a neighbourhood of 
$y$ we have 
$$
K_G+tC|_G=(K_{\PP^d}+\Theta+tC)|_G. 
$$ 
This proves the existence of $t$.\\  

\emph{Step 4.} 
Letting $\epsilon'=\frac{t}{2}\epsilon$, 
in this step we construct a boundary $\Delta$ such that $(\PP^d,\Delta)$ is $\epsilon'$-lc, 
$K_{\PP^d}+\Delta\sim_\R 0$, and $a(R,\PP^d,\Delta)\le 1$. 
We start with 
an auxiliary boundary $D=(1-\frac{t}{2})\Theta+\frac{t}{2}C$.  
We will show that $(\PP^d,D)$ is $\epsilon'$-lc. 
Let $E$ be a prime divisor over $\PP^d$ and let $I$ be its centre on $\PP^d$. If $I$ passes 
through $z$, then by Lemma \ref{l-average-boundary},
$$
a(E,\PP^d,D)=(1-\frac{t}{2})a(E,\PP^d,\Theta)+\frac{t}{2}a(E,\PP^d,C)\ge \frac{t}{2}\epsilon=\epsilon'
$$ 
because $(\PP^d,C)$ is $\epsilon$-lc near $z$.
So assume $z\notin I$. In particular, $I$ is not a stratum of $(\PP^d,\Theta)$ because each stratum contains $z$. 
Thus $I$ is not an lc centre of $(\PP^d,\Theta)$. Then 
$$
a(E,\PP^d,D)\ge a(E,\PP^d,\Theta+\frac{t}{2}C)=
\frac{1}{2}a(E,\PP^d,\Theta+tC)+\frac{1}{2}a(E,\PP^d,\Theta)\ge \frac{1}{2}\ge \epsilon'
$$
because $(\PP^d,\Theta+tC)$ is lc near $I$, by Step 3, and because $a(E,\PP^d,\Theta)\ge 1$ as $E$ is not an 
lc place of $(\PP^d,\Theta)$.
Therefore, we have proved  $(\PP^d,D)$ is $\epsilon'$-lc.

By Step 2, 
$$
a(R,\PP^d,D)=(1-\frac{t}{2})a(R,\PP^d,\Theta)+\frac{t}{2}a(R,\PP^d,C)=\frac{t}{2}a(R,\PP^d,C) \le 1.
$$ 
Moreover, taking $t$ small enough in Step 3 we can assume 
$$
-(K_{\PP^d}+D)=-(1-\frac{t}{2})(K_{\PP^d}+\Theta)-\frac{t}{2}(K_{\PP^d}+C)
$$ 
is ample because $\deg_{H_i}-(K_{\PP^d}+\Theta)=1$ and because
$$
\deg_{H_i} -(K_{\PP^d}+C)= d+1-\deg_{H_i} C\ge d+1-r.
$$ 
Therefore, there is a boundary $\Delta\ge D$  so that 
$(\PP^d,\Delta)$ is $\epsilon'$-lc, $K_{\PP^d}+\Delta\sim_\R 0$, and $a(R,\PP^d,\Delta)\le 1$.\\

\emph{Step 5.}
In this step we show that there is a bounded number $n\in \N$ such that $K_{\PP^d}$ has a 
klt $n$-complement $K_{\PP^d}+\Omega$ with $a(R,\PP^d,\Omega)\le 1$.
Since $W\to \PP^d$ is a sequence of blowups toroidal with respect to $(\PP^d,\Theta)$, 
it is a sequence of toric blowups and $W$ is a toric variety. 
Let  $\psi\colon W'\to \PP^d$ be the weighted blowup given by $R$, that is, $\psi$ is an extremal 
birational contraction contracting a single prime divisor which is the birational transform of $R$.  
Let $K_{W'}+\Delta_{W'}$ be the pullback of $K_{\PP^d}+\Delta$. Then $\Delta_{W'}$ is effective 
as $a(R,\PP^d,\Delta)\le 1$. Moreover, $-K_{W'}$ is big by construction. 

Now run an MMP on $-K_{W'}$ and let $W''$ be the resulting model. Then $W''$ is a toric variety, 
and $-K_{W''}$ is nef and big. Thus $-K_{W''}$ defines a contraction $W''\to W'''$ to a toric Fano variety. 
Moreover, since  
$(W',\Delta_{W'})$ is $\epsilon'$-lc, $(W''',\Delta_{W'''})$ is $\epsilon'$-lc too. 
Thus $W'''$ is an $\epsilon'$-lc toric Fano variety. 
Now by [\ref{A-L-Borisov}], $W'''$ belongs to a bounded family of varieties 
depending only on $d,\epsilon'$. 
Therefore, there is a natural number $n>1$ depending only on $d,\epsilon'$ such that $|-nK_{W'''}|$ is base point free, 
in particular, $K_{W'''}$ has an $n$-complement $K_{W'''}+\Omega_{W'''}$ which is klt. This gives an $n$-complement 
$K_{W''}+\Omega_{W''}$ of $K_{W''}$ which in turn gives an $n$-complement 
$K_{W'}+\Omega_{W'}$ of $K_{W'}$ because $W'\bir W''$ is an MMP on $-K_{W'}$. Then   
we get an $n$-complement $K_{\PP^d}+\Omega$ of $K_{\PP^d}$
which is klt. By construction, 
$$
a(R,\PP^d,\Omega)=a(R,{W'},\Omega_{W'})\le 1.
$$\ 

\emph{Step 6.}
In this step we finish the proof by applying Lemma \ref{l-coeff-tor-blowup}.
Since $n\Omega$ is integral and $\deg_{H_i}\Omega=d+1$, the pair $(\PP^d,\Supp (\Omega+\Theta))$ belongs to 
a bounded family of pairs depending only on $d,n$. Therefore, there is a positive real number $u$ depending only on $d,n$ 
such that $(\PP^d,\Omega+u\Theta)$ is lc. Since $\Omega_{W'}\ge 0$, we deduce that 
the coefficient of the birational transform of $R$ in  $\psi^*\Theta$ is at most $\frac{1}{u}$ which in turn implies 
$\mu_R\phi^*\Theta\le\frac{1}{u}$ where $\phi$ denotes $W\to \PP^d$. 
On the other hand, since $W\to \PP^d$ is a sequence of centre blowups of $R$ which is toroidal with respect to $(\PP^d,\Theta)$, 
we have  $l+1\le \mu_R\phi^*\Theta$, by Lemma \ref{l-coeff-tor-blowup}. 
Therefore, $l\le p:=\rddown{\frac{1}{u}-1}$.

\end{proof}

\subsection{Bound on multiplicity at an lc place}
The next result bounds the multiplicity of divisors at lc places of a pair under suitable 
assumptions. An example of such boundedness is the boundedness of $\mu_R\phi^*\Theta$ that we obtained in Step 6 of the 
previous proof.

\begin{prop}\label{p-bnd-lct-toroidal-2}
Let $d,r,n$ be natural numbers and $\epsilon$ be a positive real number. 
Assume Theorem \ref{t-bnd-lct} holds in dimension $\le d-1$.
Then there is a positive number $q$ depending only on $d,r,n,\epsilon$ satisfying the following. 
Assume 
\begin{itemize}
\item  $(X,B)$ is a projective $\epsilon$-lc pair of dimension $d$, 

\item  $A$ is a very ample divisor on $X$ with $A^d\le r$, 

\item   $\Lambda\ge 0$ is a $\Q$-divisor on $X$ with $n\Lambda$ integral, 

\item  $L\ge 0$ is an $\R$-divisor on $X$,

\item  the divisors 
$$
\mbox{$A-B$,  $~~A-\Lambda$, and $~~A - L$}
$$ 
are all ample (in particular, we are assuming $B,\Lambda,L$ are all $\R$-Cartier),

\item $(X,\Lambda)$ is lc near a point $x$ (not necessarily closed), 

\item  $T$ is an lc place of $(X,\Lambda)$ with centre the closure of $x$, and

\item  $a(T,X,B)\le 1$.\\
\end{itemize}
Then for any resolution $\nu\colon U\to X$ so that $T$ is a divisor on $U$, we have $\mu_T\nu^*L\le q$.
\end{prop}

We first treat a special case of the proposition. 

\begin{lem}\label{lem-1-for-p-bnd-lct-toroidal-2}
Assume that Proposition \ref{p-bnd-lct-toroidal-2} holds in dimension $d-1$. 
Then the proposition holds in dimension $d$ when $x$ is not a closed point. 
\end{lem}
\begin{proof}
Let $C$ be the closure of $x$. Take a general $H\in |A|$ and 
let 
$$
\mbox{$B_H=B|_H$, $~~A_H=A|_H$, $~~\Lambda_H=\Lambda|_H$, and $~~L_H=L|_H$.}
$$ 
Then by adjunction
$$
K_H+B_H=(K_X+B+H)|_H ~~\mbox{and}~~ K_H+\Lambda_H=(K_X+\Lambda+H)|_H. 
$$
Now take a log resolution 
$\phi\colon W\to X$ on which $T$ is a divisor and let $G=\phi^*H$. 
Pick a component $S$ of $G\cap T$ and let $R$ be its image on $X$. Then $R$ is contained in $H\cap C$.
In fact, considering the map $T\to C$ and taking into account the generality of $H$, we can assume that $R$ 
is a component of $H\cap C$. 

Now we have: 
\begin{itemize}
\item  $(H,B_H)$ is a projective $\epsilon$-lc pair of dimension $d-1$, 

\item  $A_H$ is a very ample divisor on $H$ with $A_H^{d-1}\le r$, 

\item   $\Lambda_H\ge 0$ is a $\Q$-divisor on $H$ with $n\Lambda_H$ integral, 

\item  $L_H\ge 0$ is an $\R$-divisor on $H$,

\item  the divisors 
$$
 \mbox{$A_H-B_H$, $~~A_H-\Lambda_H$, and $~~A_H - L_H$}
$$
are ample,

\item $(H,\Lambda_H)$ is lc near the generic point of $R$ (by the generality of $H$), 

\item  $S$ is an lc place of $(H,\Lambda_H)$ with centre $R$, and

\item  $a(S,H,B_H)\le 1$.
\end{itemize}

The last two points can be seen by considering the pullbacks $K_W+\Lambda_W$ and $K_W+B_W$ of  
$K_X+\Lambda$ and $K_X+B$, respectively, and noting that 
$$
(K_W+\Lambda_W+G)|_G ~~\mbox{and}~~ (K_W+B_W+G)|_G
$$
are the pullbacks of 
$$
K_H+\Lambda_H ~~\mbox{and}~~ K_H+B_H
$$
respectively, and that $\mu_T\Lambda_W=1$ while $\mu_TB_W\ge 0$.

Denote the induced contraction $G\to H$ by $\sigma$. 
Since we are assuming that Proposition \ref{p-bnd-lct-toroidal-2} holds in dimension $d-1$,  
the coefficient of $S$ in 
$\sigma^*(L_H)=(\phi^*L)|_G$ is bounded from above which implies the coefficient of $T$ in 
$\phi^*L$ is bounded from above too. Finally note that $\mu_T\phi^*L=\mu_T\nu^*L$ 
for any resolution $\nu\colon U\to X$ on which $T$ is a divisor.

\end{proof}

\begin{lem}\label{lem-2-for-p-bnd-lct-toroidal-2}
Assume that Proposition \ref{p-bnd-lct-toroidal-2} holds in dimension $d-1$. 
Then the proposition holds in dimension $d$ when $(X,\Lambda)$ is log smooth and $\Lambda$ is reduced.
\end{lem}
\begin{proof}
We will assume $d>1$ as the case $d=1$ holds trivially. We will use Lemma \ref{lem-1-for-p-bnd-lct-toroidal-2} 
so that we can assume $x$ is a closed point. The idea is then to apply Proposition \ref{p-bnd-lct-toroidal-1}. 
For this we need to modify the setting so that  
$\Supp B$ does not contain any stratum of $(X,\Lambda)$ other than $x$, and this occupies much of the proof.\\

\emph{Step 1.}
In this step we reduce the proposition to the case in which $x$ is a closed point and that 
$(X,\Lambda)$ has no other zero-dimensional stratum. Applying Lemma \ref{lem-1-for-p-bnd-lct-toroidal-2}, 
we can indeed assume that $x$ is a closed point.
Assume $(X,\Lambda)$ has another zero-dimensional 
stratum $y\neq x$. Let $X'\to X$ be the blowup of $X$ at $y$ and $E'$ be the exceptional divisor. Let
$K_{X'}+B'$, $K_{X'}+\Lambda', L'$, be the pullbacks of $K_X+B$, $K_X+\Lambda$, $L$, respectively.
Then
$$
\mu_{E'}B'\ge -d+1 ~~ \mbox{and} ~~\mu_{E'}\Lambda'=1.
$$ 
We can pick a 
very ample divisor $A'$ on $X'$ with bounded $A'^d$ such that 
$$
A'-B', ~~~A'-\Lambda', ~~~A'-L'
$$ 
are all ample.

Now there is $\beta\in(0,1)$ depending only on $d$ such that 
$$
B'':=\beta B'+(1-\beta)\Lambda'\ge 0.
$$   
Let $\Lambda''$ be the birational transform of $\Lambda$ and let $\epsilon'=\epsilon \beta$. 
 Replacing $A'$ with a bounded multiple we can assume 
$$
A'-B'', ~~~A'-\Lambda'', ~~~A'-L'
$$ 
are all ample, by definition of $B'',\Lambda''$.

Note that each stratum of $(X',\Lambda'')$ is the birational transform of a stratum of 
$(X,\Lambda)$. Now replacing $X,B,A,\Lambda, L,\epsilon$ with $X',B'',A',\Lambda'',L',\epsilon'$, and replacing 
$r$ appropriately, we can remove one of the zero-dimensional strata of $(X,\Lambda)$ other than $x$. 
Repeating this process a bounded number of times, we get to the situation in which $(X,\Lambda)$ 
has no zero-dimensional stratum other than $x$.\\

\emph{Step 2.}
In this step we find a positive real number $t$ depending only on $d,r,\epsilon$ such that 
$(X,B+tB)$ is $\frac{\epsilon}{2}$-lc outside finitely many closed points. 
Let $H$ be a general element of $|A|$ and let $A_H=A|_H$, $B_H=B|_H$, and $L_H=L|_H$. Then
\begin{itemize}
\item  $(H,B_H)$ is a projective $\epsilon$-lc pair of dimension $d-1$, 

\item  $A_H$ is a very ample divisor on $H$ with $A_H^{d-1}\le r$, and

\item  $A_H-B_H$ is ample.
\end{itemize}  
Thus since we are assuming Theorem \ref{t-bnd-lct} 
in dimension $\le d-1$, we find a positive real number $t$ 
depending only on $d,r,\epsilon$ such that $(H,B_H+2tB_H)$ is klt. This implies that $(X,H+B+2tB)$ is 
plt in a neighbourhood of $H$, by inversion of adjunction [\ref{kollar-mori}, Theorem 5.50] (note that 
[\ref{kollar-mori}, Theorem 5.50] assumes $B$ to be a $\Q$-divisor but the conclusion holds for 
$\R$-divisors as well). Therefore, $(X,B+2tB)$ is klt near $H$, hence 
$(X,B+2tB)$ is klt outside finitely many closed points because $H$ being general very ample it intersects 
every positive dimensional subvariety of $X$. 
In particular, $(X,B+tB)$ is $\frac{\epsilon}{2}$-lc
outside these finitely many closed points, by Lemma \ref{l-average-boundary}, because 
$$
K_X+B+tB=\frac{1}{2}(K_X+B)+\frac{1}{2}(K_X+B+2tB)
$$
and because $(X,B)$ is $\epsilon$-lc.\\

\emph{Step 3.}
In this step we take a resolution and introduce a boundary $\Gamma_V$.
Let $\psi \colon V\to X$ be a log resolution of $(X,B)$ on which $T$ is a divisor. 
Define a boundary 
$$
\Gamma_V=(1+t)B^{\sim}+(1-\frac{\epsilon}{4})\sum E_i+(1-a)T
$$
where $E_i$ are the exceptional divisors of $\psi$ other than $T$, $a=a(T,X,B)$, 
and $\sim$ denotes birational transform. Let $a_i=a(E_i,X,B)$. 
Since $(X,B)$ is $\epsilon$-lc and since $\mu_T\Gamma_V=1-a$, we have 
$$
K_V+\Gamma_V=K_V+B^\sim+\sum (1-a_i)E_i+(1-a)T+tB^\sim+\sum (a_i-\frac{\epsilon}{4})E_i
$$
$$
=\psi^*(K_X+B)+tB^{\sim}+F
$$
where 
$$
F:=\sum (a_i-\frac{\epsilon}{4})E_i
$$ 
is effective and exceptional over $X$ and its support does not contain $T$.
On the other hand, if 
$$
\mbox{$a'=a(T,X,B+tB)~~~$ and $~~~a_i'=a(E_i,X,B+tB)$},
$$ 
then we can write 
$$
K_V+\Gamma_V=K_V+(1+t)B^\sim+\sum (1-a_i')E_i+(1-a')T+\sum (a_i'-\frac{\epsilon}{4})E_i+(a'-a)T
$$
$$
=\psi^*(K_X+(1+t)B)+G
$$
where 
$$
G:=\sum (a_i'-\frac{\epsilon}{4})E_i+(a'-a)T
$$ 
is exceptional over $X$. Moreover, if  the 
image of $E_i$ on $X$ is positive-dimensional for some $i$, then $E_i$ is a component of $G$ with positive coefficient 
because $(X,B+tB)$ is $\frac{\epsilon}{2}$-lc outside finitely many closed points.\\ 

\emph{Step 4.}
In this step we run an MMP on $K_V+\Gamma_V$ over $X$ and study its outcome.
By construction, $(V,\Gamma_V)$ is klt.
Run an MMP on $K_V+\Gamma_V$ over $X$, let $Y$ be the resulting model, and  $\pi\colon Y\to X$ 
the corresponding morphism. Since $K_V+\Gamma_V\equiv G/X$ and $G$ is exceptional$/X$,
 the MMP contracts every component of $G$ with positive coefficient, by the negativity lemma. Thus 
$\pi$ is an isomorphism over the complement of finitely many closed points: indeed, by the last sentence of the previous 
step every $E_i$ with positive-dimensional centre on $X$ is a component of $G$ with 
positive coefficient so these are all contracted; also note that $X$ is smooth by assumption 
so it is $\Q$-factorial, hence the claim. 
Moreover, since 
$$
K_V+\Gamma_V\equiv tB^{\sim}+F/X
$$ 
and since $T$  is not a component of $tB^{\sim}+F$, we see that $T$ is not contracted by the MMP.\\

\emph{Step 5.}
In this step we introduce a divisor $D_Y$.
Let $A_Y$ be the pullback of $A$ to $Y$. 
By boundedness of the length of extremal rays [\ref{kawamata-bnd-ext-ray}] and by the base point free theorem, 
$K_{Y}+\Gamma_{Y}+3dA_{Y}$ 
is nef and big and semi-ample, globally. Pick a general 
$$
0\le D_{Y}\sim_\R \frac{1}{t}(K_{Y}+\Gamma_{Y}+3dA_{Y})
$$
with coefficients $\le 1-\epsilon$.
By Step 3, 
$$
K_{Y}+\Gamma_{Y}=\pi^*(K_X+B)+tB^{\sim}_Y+F_{Y},
$$
where $B^{\sim}_Y,F_Y$ denote the pushdowns of $B^{\sim},F$, respectively. 
Thus we get 
$$
\frac{1}{t}(K_{Y}+\Gamma_{Y}+3dA_{Y})=\frac{1}{t}\pi^*(K_X+B+3dA)+B^{\sim}_Y+\frac{1}{t}F_{Y}.
$$
Then we can write 
$$
D_Y=\pi^*H+B^\sim_Y+\frac{1}{t}F_Y
$$ 
for some 
$$
H\sim_\R \frac{1}{t}(K_X+B+3dA).
$$ 
Letting $D$ be the pushdown of $D_Y$ to $X$, we get $D=H+B$.\\

\emph{Step 6.}
In this step replacing $B$ with $D$ we reduce the proposition to the situation in which 
$\Supp B$ does not contain any stratum of $(X,\Lambda)$ other than $x$. 
First we show that $(X,D)$ is $\epsilon$-lc.
Write 
$$
K_Y+B^\sim_Y+R_Y=\pi^*(K_X+B)
$$
where $R_Y$ is exceptional over $X$. We then have 
$$
K_Y+R_Y+D_Y-\frac{1}{t}F_Y=K_Y+R_Y+B^\sim_Y+\pi^*H
$$
$$
=\pi^*(K_X+B+H)=\pi^*(K_X+D).
$$
Since $(X,B)$ is $\epsilon$-lc, $(Y,B^\sim+R_Y)$ is sub-$\epsilon$-lc, hence 
$$
(Y,R_Y+D_Y-\frac{1}{t}F_Y)
$$ 
is sub-$\epsilon$-lc because $D_Y$ is general semi-ample 
with coefficients $\le 1-\epsilon$ and $F_Y\ge 0$.  Therefore, $(X,D)$ is $\epsilon$-lc.  
Moreover, since $T$ is not a component of $B^\sim_Y+F_Y+D_Y$,
$$
a(T,X,D)=1-\mu_T(R_Y+D_Y-\frac{1}{t}F_Y)=1-\mu_TR_Y
$$
$$
=1-\mu_T(B^\sim+R_Y)=a(T,X,B)\le 1.
$$
 On the other hand, since 
$$
D\sim_\R \frac{1}{t}(K_X+B+3dA)+B,
$$
 there is a natural number $m$ depending only on $d,r,t$, so depending only on $d,r,\epsilon$, 
such that 
$$
3mA- D\sim_\R (mA-\frac{1}{t}K_X)+(mA-\frac{3d}{t}A)+(mA-\frac{1+t}{t}B)
$$ 
is ample because we can ensure that each bracket is an ample divisor. 

Since $\pi$ is an isomorphism over the complement of finitely many closed points and since $D_Y$ 
is semi-ample, we can assume that 
$\Supp D$ does not contain any positive-dimensional stratum of $(X,\Lambda)$.
Replacing $B$ with $D$, and then replacing $A$ with $3mA$ and replacing $r$ accordingly, 
we can assume that $\Supp B$ does not contain any positive-dimensional stratum of $(X,\Lambda)$. 
Since $x$ is the only zero-dimensional stratum of $(X,\Lambda)$, by Step 1, $\Supp B$ 
does not contain any stratum other than $x$.\\

\emph{Step 7.}
In this step we apply Proposition \ref{p-bnd-lct-toroidal-1} and finish the proof. All the assumptions of 
Proposition \ref{p-bnd-lct-toroidal-1} are satisfied in our setting, so there is a natural number $p$ depending only on 
$d,r,\epsilon$ such that $T$ can be obtained 
by a sequence of centre blowups
$$
\nu\colon U=X_l\to \cdots \to X_0=X,
$$ 
toroidal with respect to $(X,\Lambda)$, 
and of length $l\le p$. 

Since $(X,\Lambda)$ belongs to a bounded family of pairs, one can show inductively that all the $X_i$ are bounded. 
Moreover, we can assume that there is a very ample divisor $J$ on $U$ such that $J^d$ is bounded 
from above depending only on 
$d,r,\epsilon$ and that $J-\nu^*A$ is ample. 
In particular, since $\nu^*A-\nu^*L$ is nef, $J-\nu^*L$ is ample. Therefore, there is 
 a natural number $q$ depending only on 
$d,r,\epsilon$  such that 
$$
\mu_T\nu^*L\le \deg_J\nu^*L \le J^d \le q.
$$ 
Thus 
$\mu_T\nu^*L \le q$ if we replace $\nu$ with any other resolution 
on which $T$ is a divisor.

\end{proof}

\begin{proof} (of Proposition \ref{p-bnd-lct-toroidal-2})
Applying induction on dimension we can assume that the proposition holds in dimension $\le d-1$.
Let $X,B,A,L$, $\Lambda,x$ be as in the proposition in dimension $d$.
Replacing $A$ we can assume it is effective and that $A-(K_X+B)$ is ample (we use the latter 
below).
 Since $A-\Lambda$ is ample, $\deg_A\Lambda<A^d\le r$. 
Thus since $n\Lambda$ is integral, 
the couple $(X, \Supp(\Lambda+A))$ belongs to a bounded family 
of couples $\mathcal{P}$ depending only on $d,r,n$. Then there exists a log
bounded  family $\mathcal{Q}$ of log smooth pairs depending only on $d,r,n$ 
such that there exist a log resolution $\phi\colon W\to X$ of $(X,\Lambda)$ and a very ample divisor $A_W\ge 0$ 
so that 
\begin{itemize}
\item if $\Theta_W$ is the sum of the exceptional divisors of $\phi$ and the 
support of the birational transform of $\Lambda$, then $(W,\Theta_W+A_W)$ belongs to $\mathcal{Q}$, 

\item $A_W-\Theta_W$ and $A_W-\phi^*A$ are ample.
\end{itemize}
In particular, this means that $A_W^d$ is bounded from above.

Let $K_W+\Lambda_W$ be 
the pullback of $K_X+\Lambda$. Since $(X,\Lambda)$ is lc near $x$, 
$\Lambda_W\le \Theta_W$ over some neighbourhood of $x$, hence 
$$
0=a(T,X,\Lambda)=a(T,W,\Lambda_W)\ge a(T,W,\Theta_W)\ge 0
$$ 
which shows that $T$ is an lc place of $(W,\Theta_W)$. Moreover, if $C$ is the centre of $T$ on $W$ 
and if $w$ is its generic point, 
then $\Lambda_W=\Theta_W$ near $w$.

Let $K_W+B_W$ be the pullback of $K_X+B$. We claim that the coefficients of $B_W$ are bounded from below.  
Writing $K_W+J_W=\phi^*K_X$, we see that $J_W\le B_W$, so it is enough to show that 
the coefficients of $J_W$ are bounded from below. Note that since both $K_X+B$ 
and $B$ are $\R$-Cartier by assumption, $K_X$ is $\Q$-Cartier, so $\phi^*K_X$ makes sense.   
Since $(X,\Supp \Lambda)$ belongs to bounded family of couples, we can indeed choose $\phi$ so that the coefficients of 
$J_W$ belong to a fixed finite set, hence in particular they are bounded from below.

Now there is a fixed $\alpha\in(0,1)$ such that 
$$
\Delta_W:=\alpha B_W+(1-\alpha)\Theta_W\ge 0
$$
because the coefficients of $B_W$ are bounded from below and each exceptional$/X$ 
prime divisor has coefficient $1$ in $\Theta_W$.  

Let $\delta= \alpha \epsilon$. Since $(W,B_W)$ is sub-$\epsilon$-lc and $(W,\Theta_W)$ 
is lc, the pair $(W,\Delta_W)$ is $\delta$-lc, by Lemma \ref{l-average-boundary}. Moreover, 
since 
$$
a(T,W,\Theta_W)=0,
$$
we have 
$$
a(T,W,\Delta_W)=\alpha a(T,W,B_W)+(1-\alpha)a(T,W,\Theta_W)
$$
$$
=\alpha a(T,W,B_W)=\alpha a(T,X,B) \le 1.
$$ 

On the other hand, by the above and by assumption, $A_W-\Theta_W$ is ample, and $\phi^*A-(K_W+B_W)$ and $\phi^*A-L_W$ are nef 
where $L_W$ is the pullback of $L$. Since $A_W-\phi^*A$ is ample, we deduce that 
$A_W-(K_W+B_W)$ and $A_W-L_W$ are ample. Moreover, replacing $A_W$ with a bounded multiple we can assume that  $\frac{1}{2}A_W+K_W$ and $\frac{1}{2}A_W-(K_W+B_W)$ are ample. In particular,  
${A_W}- B_W$ is ample, hence 
$$
{A_W} -\Delta_W=\alpha (A_W-B_W)+(1-\alpha)(A_W-\Theta_W)
$$ 
is ample too.

Now after replacing $\epsilon$ with $\delta$ 
and replacing $r$ appropriately, we can replace $X,B,A,L$, $\Lambda,x$ with 
$W,\Delta_W,A_W,L_W,\Theta_W,w$, respectively. In particular, we can assume that $(X,\Lambda)$ is log smooth 
with $\Lambda$ reduced. We are then done by Lemma \ref{lem-2-for-p-bnd-lct-toroidal-2}.

\end{proof}

\subsection{Construction of $\Lambda$}

We want to apply Proposition \ref{p-bnd-lct-toroidal-2} to prove 
Theorem \ref{t-bnd-lct}. The proposition assumes existence of an extra divisor 
$\Lambda$ which helps to eventually reduce the problem to the case of toric varieties via \ref{p-bnd-lct-toroidal-1}. 
Next, we will use complements (as in Theorem \ref{t-compl-near-lcc}) to get the required divisor $\Lambda$.

\begin{prop}\label{p-finding-Lambda}
Let $d,r$ be natural numbers and $\epsilon$ be a positive real number. 
Assume Theorem \ref{t-bnd-lct} holds in dimension $\le d-1$.
Then there exist natural numbers $n,m$ and a positive real number $\epsilon'<\epsilon$  
depending only on $d,r,\epsilon$ satisfying the following. 
Assume 
\begin{itemize}

\item $(X,B)$ is a projective $\epsilon$-lc pair of dimension $d$, 

\item $A$ is a very ample divisor on $X$ with $A^d\le r$, 

\item $L\ge 0$ is an $\R$-divisor on $X$, 

\item $A -B$ and $A - L$ are ample (so we are assuming that $B$ and $L$ are $\R$-Cartier), 

\item $(X,B+tL)$ is $\epsilon'$-lc for some $t\in (0,r)$, 

\item we have 
$$
a(T,X,B+tL)=\epsilon'
$$ 
for some prime divisor $T$ over $X$, and

\item the centre of $T$ on $X$ is a closed point $x$.\\
\end{itemize}
Then there is a $\Q$-Cartier $\Q$-divisor $\Lambda\ge 0$ such that 
\begin{itemize}

\item $n\Lambda$ is integral, 

\item $mA- \Lambda$ is ample, 

\item $(X,\Lambda)$ is lc near $x$,  and 

\item $T$ is an lc place of $(X,\Lambda)$.
\end{itemize}
\end{prop}
\begin{proof}
We can assume that $d>1$ as the proposition holds trivially in dimension one.
On the other hand, by the ACC for lc thresholds [\ref{HMX2}], there exists a real 
number $\epsilon'\in(0,\epsilon)$,  
depending only on $d$, so that if $(Y,(1-\epsilon')S)$  is a 
 klt pair of dimension $d$ where $S$ is reduced and $\Q$-Cartier, then $(Y,S)$ is lc. 
Let $X,B,A,L,T,t$ be as in the proposition with $\epsilon'$ as in the previous sentence.\\

\emph{Step 1.}
In this step we find a positive real number $v$ 
depending only on $d,r,\epsilon$ such that
$$
(X,B+tL+v(B+tL))
$$ 
is $\frac{\epsilon'}{2}$-lc outside finitely many closed points.
Since $t< r$,
$$
(r+1)A-B-tL=A-B+t(A-L)+(r-t)A
$$ 
is ample, so replacing $A$ with $(r+1)A$ and replacing $r$ accordingly we can assume that $A-B-tL$ is ample. 

Let $H$ be a general element of $|A|$ and let $A_H=A|_H$, $B_H=B|_H$, and $L_H=L|_H$. Then 
\begin{itemize}
\item  $(H,B_H+tL_H)$ is a projective $\epsilon'$-lc pair of dimension $d-1$, 

\item  $A_H$ is a very ample divisor on $H$ with $A_H^{d-1}\le r$, and

\item  $A_H-B_H-tL_H$ is ample.
\end{itemize}  
Thus since we are assuming Theorem \ref{t-bnd-lct} 
in dimension $\le d-1$, we find a positive real number $v$ 
depending only on $d,r,\epsilon'$ such that 
$$
(H,B_H+tL_H+2v(B_H+tL_H))
$$ 
is klt. As $\epsilon'$ was picked depending only on $d$, we can choose $v$ depending only on $d,r,\epsilon$. 
Now
$$
(X,H+B+tL+2v(B+tL))
$$ 
is plt in a neighbourhood of $H$, by inversion of adjunction [\ref{kollar-mori}, Theorem 5.50]. 
Therefore, 
$$
(X,B+tL+2v(B+tL))
$$ 
is klt near $H$, hence it is klt outside finitely many closed points. 
In particular, 
$$
(X,B+tL+v(B+tL))
$$ 
is $\frac{\epsilon'}{2}$-lc outside finitely many closed points, 
by Lemma \ref{l-average-boundary}, because 
$$
K_X+B+tL+v(B+tL)=\frac{1}{2}(K_X+B+tL)+\frac{1}{2}(K_X+B+tL+2v(B+tL))
$$
and because $(X,B+tL)$ is $\epsilon'$-lc.\\

\emph{Step 2.}
In this step we take a log resolution $W\to X$ and define a boundary $\Gamma_W$ and study some of its properties.
Let $\psi\colon W\to X$ be a log resolution of $(X,B+tL)$ so that $T$ is a divisor on $W$.
Define a boundary 
$$
\Gamma_W=(1+v)(B^{\sim}+tL^{\sim})+(1-\frac{\epsilon'}{4})\sum E_i+(1-\epsilon')T
$$
where $E_i$ are the exceptional divisors of $\psi$ other than $T$, and $\sim$ denotes birational transform. 
Let 
$$
c_i=a(E_i,X,B+tL).
$$ 
Since 
$$
\epsilon'=a(T,X,B+tL),
$$
we can write  
$$
K_W+\Gamma_W
=K_W+B^\sim+tL^\sim+\sum (1-c_i)E_i+(1-\epsilon')T+v(B^\sim+tL^\sim)+\sum (c_i-\frac{\epsilon'}{4})E_i
$$
$$
=\psi^*(K_X+B+tL)+v(B^{\sim}+tL^{\sim})+F
$$
where 
$$
F:=\sum (c_i-\frac{\epsilon'}{4})E_i
$$ 
is exceptional over $X$ and its support does not contain $T$.
Moreover, since $(X,B+tL)$ is $\epsilon'$-lc, we have $c_i\ge \epsilon'$ for every $i$, so $F$ is effective. 

On the other hand, letting 
$$
c_i'=a(E_i,X,(1+v)(B+tL))
$$ 
and 
$$
a'=a(T,X,(1+v)(B+tL))
$$
we can write 
$$
K_W+\Gamma_W
=K_W+(1+v)(B^\sim+tL^\sim)+\sum (1-c_i')E_i+(1-a')T+\sum (c_i'-\frac{\epsilon'}{4})E_i+(a'-\epsilon')T
$$
$$
=\psi^*(K_X+(1+v)(B+tL))+G
$$
where 
$$
G:=\sum (c_i'-\frac{\epsilon'}{4})E_i+(a'-\epsilon')T
$$ 
is exceptional over $X$. Moreover, since
$$
(X,(1+v)(B+tL))
$$ 
is $\frac{\epsilon'}{2}$-lc outside finitely many closed points, 
if the image of $E_i$ on $X$ is positive-dimensional for some $i$, 
then $c_i'\ge \frac{\epsilon'}{2}$ so $E_i$ is a component of $G$ with positive coefficient.\\

\emph{Step 3.}
In this step we run an MMP on $K_W+\Gamma_W$ over $X$ and argue that it does not contract $T$.
By construction, $(W,\Gamma_W)$ is klt; here we are using the fact that the coefficients of
 $(1+v)(B^{\sim}+tL^{\sim})$ are less than $1$ because 
$$
(X,(1+v)(B+tL))
$$ 
is $\frac{\epsilon'}{2}$-lc outside finitely many closed points.

Run an MMP on $K_W+\Gamma_W$ over $X$ and let $Y'$ be the resulting model. Since 
$K_W+\Gamma_W\equiv G/X$ and since $G$ is exceptional over $X$ by Step 2, 
applying the negativity lemma shows that the MMP contracts every component of $G$ with positive coefficient. 
In particular, every $E_i$ with positive-dimensional image on $X$ is contracted, by Step 2. 
Thus the induced morphism $Y'\to X$ is a small morphism over the complement of finitely many closed points. 
On the other hand, 
by Step 2,   
$$
K_W+\Gamma_W\equiv v(B^{\sim}+tL^{\sim})+F/X 
$$
and $T$ is not a component of 
$$
v(B^{\sim}+tL^{\sim})+F\ge 0,
$$
 so the MMP does not contract $T$.

Let $A_{Y'}$ be the pullback of $A$. 
By boundedness of the length of extremal rays [\ref{kawamata-bnd-ext-ray}] and by the 
base point free theorem, $K_{Y'}+\Gamma_{Y'}+2dA_{Y'}$ 
is semi-ample, globally. Note that over the complement of a finite set of closed points on $X$, 
$K_{Y'}+\Gamma_{Y'}+2dA_{Y'}$ is the pullback of 
$$
K_X+(1+v)(B+tL)+2dA.
$$\ 

\emph{Step 4.}
In this step we consider a model $Y$ on which $-T_{Y}$ is ample over $X$. 
 Since $(X,B+tL)$ is $\epsilon'$-lc and 
$$
a(T,X,B+tL)=\epsilon'<1,
$$
there is a birational contraction $Y''\to X$ extracting $T$ but no other divisor where $Y''$ is normal and $\Q$-factorial. 
We denote the centre of $T$ on $Y''$ by $T_{Y''}$.
Then $Y''$ is of Fano type over $X$, so there is an ample model $Y$ of $-T_{Y''}$ 
over $X$: that is, we run an MMP on $-T_{Y''}$ to get a semi-ample over $X$ divisor and 
then take the associated contraction to get $Y$. The induced birational morphism 
$\phi\colon Y\to X$ contracts $T_Y$ but no other divisors and any curve contracted 
by $\phi$ is inside $T_Y$ as $-T_Y$ is ample over $X$. 
In particular, since $T_{Y}$ is mapped to $x$, $\phi$ is an isomorphism 
over the complement of $x$ in $X$. 
By construction, the induced map $Y\bir Y'$ 
does not contract any divisor.\\

\emph{Step 5.}
In this step we show that $K_Y+\Gamma_Y+3dA_Y$ is ample where 
 $K_Y+\Gamma_Y,A_Y$ are the pushdowns of $K_{Y'}+\Gamma_{Y'},A_{Y'}$. 
  First note that, by Step 2, 
$$
K_{Y}+\Gamma_Y\sim_\R G_Y=(a'-\epsilon')T_Y/X
$$
which in particular means that $K_Y+\Gamma_Y$ is $\R$-Cartier. It is obvious that $A_Y$ is 
$\Q$-Cartier. Thus $K_Y+\Gamma_Y+3dA_Y$ is $\R$-Cartier.

 Since 
$$
a(T,X,B)\ge \epsilon > \epsilon'= a(T,X,B+tL)
$$
we get 
$$
\mu_{T_Y}\phi^*tL\ge \epsilon-\epsilon'>0.
$$ 
Thus 
$$
a'=a(T,X,(1+v)(B+tL))<\epsilon'=a(T,X,B+tL).
$$
Therefore, from 
$$
K_Y+\Gamma_Y+2dA_Y\sim_\R (a'-\epsilon')T_Y/X
$$ 
we deduce that $K_Y+\Gamma_Y+2dA_Y$ is ample over $X$ because $-T_Y$ is ample over $X$.

Now let $C$ be a curve on $Y$ that is not contracted over $X$. Let $c$ be a general closed point of $C$.
By Steps 2 and 3, 
$$
K_{Y'}+\Gamma_{Y'}+2dA_{Y'}\sim_\R G_{Y'}/X
$$
and $G_{Y'}=0$ over the complement of finitely many closed points. Thus since 
$K_{Y'}+\Gamma_{Y'}+2dA_{Y'}$ is semi-ample by Step 3 and since it is $\sim_\R 0$ over a neighbourhood 
of $\phi(c)$, we can find 
$$
0\le P_{Y'}\sim_\R K_{Y'}+\Gamma_{Y'}+2dA_{Y'}
$$
such that $P_{Y'}$ does not intersect the fibre of $Y'\to X$ over $\phi(c)$. So the pushdown of $P_{Y'}$ 
to $X$ does not contain $\phi(c)$.  
On the other hand, by Step 4, $\phi\colon Y\to X$ is an isomorphism over $\phi(c)$, so $P_Y$ 
does not contain $c$ where $P_Y$ is the pushdown of $P_{Y'}$ to $Y$. 
Then  
$$
(K_Y+\Gamma_Y+2dA_Y)\cdot C=P_Y\cdot C\ge 0.
$$ 
Therefore, $K_Y+\Gamma_Y+2dA_Y$ is nef globally while being ample over $X$. Since $A_Y$ is the pullback of the 
ample divisor $A$, it follows that $K_Y+\Gamma_Y+3dA_Y$ is ample.\\
 
\emph{Step 6.} 
In this step we show that there is a natural number $l$ depending only on $d,r,\epsilon$ such that 
$$
lA_Y-(K_Y+(1-\epsilon')T_Y)
$$ 
is ample. Since 
$$
a(T,X,B+tL)=\epsilon',
$$
we have 
$$
K_Y+B^{\sim}_Y+tL^{\sim}_Y+(1-\epsilon')T_Y=\phi^*(K_X+B+tL),
$$
where $B^{\sim}_Y,L^{\sim}_Y$ are the pushdowns of $B^{\sim},L^{\sim}$.
Hence for any $l$ we have
$$
lA_Y-(K_Y+(1-\epsilon')T_Y)=lA_Y-\phi^*(K_X+B+tL)+B^{\sim}_Y+tL^{\sim}_Y
$$
$$
=(l-\frac{3d}{v})A_Y-(1+\frac{1}{v})\phi^*(K_X+B+tL)+
\frac{1}{v}\phi^*(K_X+B+tL)+B^{\sim}_Y+tL^{\sim}_Y+\frac{3d}{v}A_Y.
$$
This in particular shows that 
$$
lA_Y-(K_Y+(1-\epsilon')T_Y)
$$
is $\R$-Cartier because $B^{\sim}_Y,L^{\sim}_Y$ are $\R$-Cartier which in turn follow from 
the fact that $\phi^*B,\phi^*L,T_Y$ are all $\R$-Cartier. 

By assumption, $A- B$ and $A- L$ are ample. Moreover, since $A$ is very ample and $A^d\le r$, 
$\beta A-K_X$ is ample for some bounded natural number $\beta$ depending only on $d,r$.
Thus we can choose $l$ depending only on 
$d,r,\epsilon$ so that 
$$
(l-\frac{3d}{v})A_Y-(1+\frac{1}{v})\phi^*(K_X+B+tL)=\phi^*((l-\frac{3d}{v})A-(1+\frac{1}{v})(K_X+B+tL))
$$
is nef where we also used the assumption $t\le r$ and that $v,\beta$ depend only on $d,r,\epsilon$. 
On the other hand, by Step 2 and the fact that $F$ is contracted over $Y$, 
$$
K_Y+\Gamma_Y=\phi^*(K_X+B+tL)+v(B^{\sim}_Y+tL^{\sim}_Y),
$$
hence by Step 5, 
$$
\frac{1}{v}(K_Y+\Gamma_Y+3dA_Y)=\frac{1}{v}\phi^*(K_X+B+tL)+B^{\sim}_Y+tL^{\sim}_Y+\frac{3d}{v}A_Y
$$
is ample. Therefore,  
$$
lA_Y-(K_Y+(1-\epsilon')T_Y)
$$ 
is ample by the previous paragraph.\\

 \emph{Step 7.}
 In this step we show that after replacing $l$ with a bounded multiple we can ensure that $lA_Y-(K_Y+T_Y)$ is  ample. 
By Step 5, we have $\mu_{T_Y}\phi^*tL\ge \epsilon-\epsilon'$. Thus there is a positive real number 
$\alpha\le \frac{\epsilon'}{\epsilon-\epsilon'}$ such that 
$$
\alpha\mu_{T_Y} \phi^*(B+tL)=\epsilon'.
$$ 
Then  
$$
\alpha \phi^*(B+tL)=\alpha(B^{\sim}_Y+tL^{\sim}_Y)+  \epsilon' T_Y.
$$
Thus we have 
$$
3lA_Y-(K_Y+T_Y)=3lA_Y-(K_Y+(1-\epsilon')T_Y)-\epsilon' T_Y
$$
$$
=3lA_Y-(K_Y+(1-\epsilon')T_Y)-\alpha \phi^*(B+tL)+\alpha(B^{\sim}_Y+tL^{\sim}_Y)
$$
$$
=(lA_Y-(K_Y+(1-\epsilon')T_Y))+(lA_Y-\alpha\phi^*(B+tL))+(lA_Y+\alpha(B^{\sim}_Y+tL^{\sim}_Y)). 
$$

We argue that we can replace $l$ with a bounded multiple so that $3lA_Y-(K_Y+T_Y)$
is ample. By Step 6, 
$$
lA_Y-(K_Y+(1-\epsilon')T_Y)
$$
is ample.  Moreover, if 
$$
l\ge \frac{(1+r)\epsilon'}{\epsilon-\epsilon'},
$$
then $l\ge {(1+t)}{\alpha}$, hence    
$$
lA_Y-\alpha\phi^*(B+tL)
$$ 
is nef because 
$$
lA-\alpha(B+tL)=(l-(1+t)\alpha) A+\alpha(A-B+tA-tL)
$$ 
is ample. In addition, by Step 6, we can write 
$$
lA_Y+\alpha(B^{\sim}_Y+tL^{\sim}_Y)
=lA_Y+\alpha (\frac{1}{v}(K_Y+\Gamma_Y+3dA_Y)-\frac{1}{v}\phi^*(K_X+B+tL)-\frac{3d}{v}A_Y)
$$
$$
=(l-\frac{3d\alpha}{v})A_Y-\frac{\alpha}{v}\phi^*(K_X+B+tL)
+\frac{\alpha}{v} (K_Y+\Gamma_Y+3dA_Y)
$$
which shows that 
$$
lA_Y+\alpha(B^{\sim}_Y+tL^{\sim}_Y)
$$
is ample if $l$ is large enough depending only on $d,r,\epsilon$ remembering from Step 5 that $K_Y+\Gamma_Y+3dA_Y$ is ample. 
Therefore, taking  $l$ large enough and then replacing it with $3l$, we can assume $lA_Y-(K_Y+T_Y)$ is ample.\\ 

\emph{Step 8.} 
In this step we finish the proof by applying Theorem \ref{t-compl-near-lcc}.
The pair $(Y,(1-\epsilon')T_Y)$ is klt because  $(X,B+tL)$ is klt and 
$$
a(T,X,B+tL)=\epsilon'.
$$ 
Thus the pair $(Y,T_Y)$ is lc, by our choice of $\epsilon'$.
Moreover, $A_Y|_{T_Y}\sim 0$ since $T_Y$ is mapped to the closed point $x$. 
Now applying Theorem \ref{t-compl-near-lcc} (by taking $M=lA_Y$, $S=T_Y$, and $z=x$), 
there is a natural number $n$ depending only on $d$ and  
there is a $\Q$-divisor $\Lambda_Y\ge T_Y$ such that 
$
(Y,\Lambda_Y)
$ 
is lc over $x$ and 
$$
n(K_Y+\Lambda_Y)\sim (n+2)lA_Y.
$$ 
Let $\Lambda$ be the pushdown of $\Lambda_Y$. Then 
$K_Y+\Lambda_Y$ is the pullback of $K_X+\Lambda$, so the pair 
$(X,\Lambda)$ is lc near $x$. Also $n\Lambda$ is integral. Moreover, from 
$$ 
K_X+\Lambda \sim_\Q \frac{(n+2)l}{n}A
$$ 
we deduce that $4lA- (K_X+\Lambda)$ is ample which in turn implies 
that there is a natural number $m$ depending only on $d,r,\epsilon$ such that $mA- \Lambda$ is ample.
Finally, 
$$
a(T,X,\Lambda)=a(T,Y,\Lambda_Y)=a(T_Y,Y,\Lambda_Y)=0,
$$ 
so $T$ is an lc place of $(X,\Lambda)$.

\end{proof}


\section{\bf Proof of main results}

We apply induction on dimension to prove  Theorems \ref{t-BAB},  \ref{t-bnd-lct-global}, and \ref{t-bnd-lct}, 
so assume they all hold in dimension $\le d-1$. It is easy to verify them in dimension one.
Recall that we proved Theorem \ref{t-compl-near-lcc} in Section 4.

\begin{proof}(of Theorem \ref{t-bnd-lct})
\emph{Step 1.}
In this step we make some simple reductions.
Since $A-B$ and $A-M$ are pseudo-effective, replacing $A$ with $2A$ we can assume $A-B$ and $A-M$ are big.
In particular, $A\sim_\R M+N$ for some $N\ge 0$. Thus   
$$
\lct(X,B,|M|_\R)\ge  \lct(X,B,|M+N|_\R)=\lct(X,B,|A|_\R),
$$
so it is enough to give a positive lower bound for the right hand side.\\

\emph{Step 2.}
In this step we reduce the theorem to the case when $K_X$ is $\Q$-Cartier.
By Lemma \ref{l-bnd-small-modification}, there exist a natural number $l$ depending only on $d,r,\epsilon$ 
such that there exist a small projective birational morphism $\phi\colon Y\to X$ and 
a very ample divisor $A_Y$ on $Y$ such that 
\begin{itemize}
\item $Y$ is normal and $K_Y$ is $\Q$-Cartier,

\item $A_Y^d\le l$ and $A_Y-\phi^*A$ is ample.
\end{itemize}
Let $K_Y+B_Y=\phi^*(K_X+B)$. Then 
$$
A_Y-B_Y=(A_Y-\phi^*A)+(\phi^*A-B_Y)
$$
is big as $\phi$ is small. Moreover, 
$$
\lct(X,B,|A|_\R)=\lct(Y,B_Y,|\phi^*A|_\R)\ge \lct(Y,B_Y,|A_Y|_\R).
$$
Thus replacing $(X,B),A,r$ with $(Y,B_Y),A_Y,l$, we can assume that $K_X$ is $\Q$-Cartier.\\ 

\emph{Step 3.}
From here to the end of Step 5 we assume that $A-B$ is ample. In Step 6 we treat the general case. 
In this step we consider the lc threshold of the $\R$-linear system 
defined by $C:=\frac{1}{2}A$ and make some preparations for applying Proposition \ref{p-finding-Lambda}. We have 
$$
\lct(X,B,|A|_\R)=\frac{1}{2}\lct(X,B,|C|_\R)
$$
because if $N\in |A|_\R$, then $L:=\frac{1}{2}N\in |C|_\R$ and 
$(X,B+tN)$ is lc if and only if $(X,B+2tL)$ is lc where $t\in \R$.
Thus it is enough to find a positive lower bound for $\lct(X,B,|C|_\R)$.

Let $n,m,\epsilon'$ be the numbers given by Proposition \ref{p-finding-Lambda} for the data $d,r,\epsilon$. 
Note that since $K_X+B$ and $K_X$ are both $\R$-Cartier, $B$ is $\R$-Cartier.
Also note that $A-C$ is ample by definition of $C$. 
Pick $L\in |C|_\R$. Let $t$ be the largest real number such that $(X,B+tL)$ is ${\epsilon'}$-lc. 
It is enough to find a positive lower bound for $t$. In particular, we can assume $t<1$.

By definition of $t$, there is a prime divisor $T$ on birational models of $X$ such that 
$$
a(T,X,B+tL)={\epsilon'}.
$$
Let $x$ be the generic point of the 
centre of $T$ on $X$.\\ 

\emph{Step 4.}
In this step we reduce to the case when $x$ is a closed point.
Assume $x$ is not a closed point. Then cutting by general elements of $|A|$ and 
applying induction (see Step 1 of the proof of Proposition \ref{p-bnd-lct-toroidal-2} 
for similar arguments), there is a positive real number $v$ depending only on $d,r,\epsilon$ such that 
$(X,B+vL)$ is lc outside finitely many closed points, in particular, it is lc near $x$. Then 
$$
(X,B+(1-\frac{\epsilon'}{\epsilon})vL)
$$ 
is $\epsilon'$-lc near $x$, 
by Lemma \ref{l-average-boundary}, because 
$$
B+(1-\frac{\epsilon'}{\epsilon})vL=\frac{\epsilon'}{\epsilon}B+(1-\frac{\epsilon'}{\epsilon})(B+vL)
$$
and because $(X,B)$ is $\epsilon$-lc.
In particular, $t\ge (1-\frac{\epsilon'}{\epsilon})v$. Thus we can assume $x$ is a closed point.\\

\emph{Step 5.}
In this step we show that $t$ is bounded from below away from zero by applying 
Propositions \ref{p-bnd-lct-toroidal-2} and \ref{p-finding-Lambda}. 
First, by Proposition \ref{p-finding-Lambda} there is a $\Q$-Cartier $\Q$-divisor $\Lambda\ge 0$ such that 
\begin{itemize}
\item $n\Lambda$ is integral, 
\item $mA -\Lambda$ is ample, 
\item $(X,\Lambda)$ is lc near $x$,  and 
\item $T$ is an lc place of $(X,\Lambda)$.
\end{itemize}  
Replacing 
$A,C,L,t$ with $2mA,2mC,2mL,\frac{t}{2m}$, respectively, and replacing  $r$ accordingly, 
we can assume $A-B-tL$ and $A-\Lambda$ are ample. Applying Proposition \ref{p-bnd-lct-toroidal-2} 
to $(X,B+tL)$, there is a natural number  $q$ depending only on $d,r,n,\epsilon'$ 
such that if $\nu\colon U\to X$ is a resolution so that $T$ is a divisor on $U$, then 
$\mu_T\nu^*L\le q$. Pick such a resolution.

Now since 
$$
a(T,X,B)\ge \epsilon>\epsilon'=a(T,X,B+tL),
$$
we have  $\mu_T\nu^*tL\ge \epsilon-\epsilon'$ which implies 
$t\ge \frac{\epsilon-\epsilon'}{q}$, hence $t$ is bounded from below as required.\\

\emph{Step 6.} 
In this step we finish the proof of the theorem. 
It remains to treat the case when $A-B$ may not be ample. We will show that after replacing 
$A$ with a bounded multiple, $A-B$ becomes ample. 
As mentioned in Step 1, 
we can assume $A-B$ is big. We can then find $0\le P\sim_\R A-B$. By Steps 1-5 above, more precisely, 
by applying the theorem in the case when the boundary is zero, we find a positive real number 
$s$ depending only on $d,r,\epsilon$ such that $(X,sP)$ is klt. Thus $K_X+sP+3dA$ is ample by 
boundedness of length of extremal rays. On the other hand, we can assume that $A-K_X$ is ample, 
hence $sP+(3d+1)A$ is ample which means that letting $e=\lceil \frac{3d+1}{s}\rceil$,  
$P+eA$ is ample. 

By assumption, $P\sim_\R A-B$, so $(e+1)A-B$ is ample. Therefore, replacing $A$ with $(e+1)A$, 
we are reduced to the case when $A-B$ is ample. 

\end{proof}

The idea of Step 6 in the previous proof is due to Yanning Xu.

\begin{proof}(of Theorem \ref{t-global-lct-attained})
This follows by combining Theorem \ref{t-bnd-lct}, Lemma \ref{l-t-global-lct-attained=1}, and 
Proposition \ref{p-t-global-lct-attained<1}.

\end{proof}

\begin{proof}(of Theorem \ref{t-bnd-lct-global})
This follows from Theorem \ref{t-bnd-lct} in dimension $\le d$, Theorem \ref{t-BAB} in dimension $\le d-1$, 
and Lemma \ref{l-local-lct-bab-to-global-lct}.

\end{proof}

\begin{proof}(of Theorem \ref{t-BAB})
Let $X'$ be a small $\Q$-factorialisation of $X$. Then $X'$ is of Fano type.
Run an MMP on $-K_{X'}$ and let $X''$ be the resulting model. Then $X''$ is an $\epsilon$-lc 
weak Fano variety because we can find $\Delta\ge B$ so that $(X,\Delta)$ is $\epsilon$-lc and $K_X+\Delta\sim_\R 0$ 
which gives $\Delta''$ so that $(X'',\Delta'')$ is $\epsilon$-lc. 
It is enough to show such $X''$ form a  bounded family because then there is a bounded 
natural number $n$ such that $K_{X''}$ has a klt $n$-complement $K_{X''}+\Omega''$ 
which gives a klt $n$-complement $K_{X'}+\Omega'$ [\ref{B-compl}, 6.1(3)] and this in turn gives a  
klt $n$-complement $K_X+\Omega$ of $K_X$, hence we can apply [\ref{HX}]. Replacing $X$ 
with $X''$ we can then assume $B=0$.

By Theorems \ref{t-eff-bir-e-lc} and \ref{t-bnd-compl}, there is a natural number $m$ 
depending only on $d,\epsilon$ such that $|-mK_{X}|$ defines a birational map and such that $K_{X}$ 
has an $m$-complement. Moreover, by Theorem \ref{t-BAB} in dimension $\le d-1$ and by Theorem \ref{t-BAB-to-bnd-vol}, 
there is a natural number $v$ depending only on $d,\epsilon$ such that $\vol(-K_{X})\le v$. 

On the other hand, by Theorem \ref{t-bnd-lct-global}, there is a positive real number $t$ depending only 
on $d,\epsilon$ such that if $0\le N\sim_\R -K_X$ then $(X,tN)$ is klt. Thus letting 
$t_l=\frac{t}{l}$ for $l\in \N$, we deduce that for any $0\le L\sim -lK_X$, the pair $(X,t_lL)$ is klt. 
Now boundedness of such $X$ follows from Theorem \ref{t-from-lct-to-bnd-var}.

\end{proof}

\begin{proof}(of Corollary \ref{cor-BAB})
Since $\Delta$ is big, we can write $\Delta\sim_\R A+D$ where $A$ is ample and $D\ge 0$. Pick 
$\alpha\in(0,1)$ and let 
$$
\Gamma=(1-\alpha)\Delta +\alpha D.
$$
Then 
$$
-(K_X+\Gamma)=-(1-\alpha)(K_X+\Delta)-\alpha (K_X+D)
$$ 
is ample because 
$$
-(K_X+D)\sim_\R \Delta-D\sim_\R A
$$
is ample. Since $(X,\Delta)$ is $\epsilon$-lc, choosing $\alpha$ to be sufficiently small we can ensure 
that $(X,\Gamma)$ is $\frac{\epsilon}{2}$-lc. 
Now apply Theorem \ref{t-BAB}.

\end{proof}

\begin{proof}(of Corollary \ref{cor-bir-aut})
This follows from Theorem \ref{t-BAB} and [\ref{Prokhorov-Shramov}, Theorem 1.8].

\end{proof}


\vspace{2cm}

\textsc{DPMMS, Centre for Mathematical Sciences} \endgraf
\textsc{University of Cambridge,} \endgraf
\textsc{Wilberforce Road, Cambridge CB3 0WB, UK} \endgraf
\email{c.birkar@dpmms.cam.ac.uk\\}


\begin{thebibliography}{99}

\bibitem{}\label{Alexeev}  
{V. Alexeev; {\emph{Boundedness and $K\sp 2$ for log surfaces.}}  Internat. J. Math.  \textbf{5}  (1994),  no. 6, 779--810.}

\bibitem{}\label{Alexeev-Brion} 
V. Alexeev, M. Brion, \emph{Boundedness of spherical Fano varieties.} 
The Fano Conference, 69-80, Univ. Torino, Turin, 2004.

\bibitem{}\label{Ambro}  
F. Ambro; \emph{Variation of Log Canonical Thresholds in Linear Systems.} 
 Int. Math. Res. Notices \textbf{14} (2016), 4418-4448. 

\bibitem{}\label{B-lcy-fibs}
C.~Birkar; \emph{Log Calabi-Yau fibrations},  arXiv:1811.10709.


\bibitem{}\label{B-compl}
C.~Birkar; \emph{Anti-pluricanonical systems on Fano varieties},  
 Ann. of Math. (2) \textbf{190} (2019), no. 2, 345-463.



\bibitem{}\label{B-lc-flips}
C.~Birkar; \emph{Existence of log canonical flips and a special LMMP},
Pub. Math. IHES., 
\textbf{115} (2012), 325-368.

\bibitem{}\label{BCHM}
C.~Birkar, P.~Cascini, C.~Hacon and J.~M$^{\rm c}$Kernan;
\emph{Existence of minimal models for varieties of log general type},
J. \ Amer. \ Math. \ Soc. \textbf{23} (2010), no. 2, 405-468.


\bibitem{}\label{A-Borisov} 
{A. Borisov;  {\emph{Boundedness of Fano threefolds with log-terminal singularities of given index.}}  
J. Math. Sci. Univ. Tokyo  \textbf{8}  (2001), no. 2, 329--342.}


\bibitem{}\label{A-L-Borisov} 
{A. Borisov, L. Borisov;  {\emph{Singular toric Fano varieties.}}  
Acad. Sci. USSR Sb. Math. \textbf{75} (1993), no. 1, 277–283.}


\bibitem{}\label{Campana}
F. Campana, \emph{Une version géométrique généralisée du théorème du produit de Nadel} 
[A generalized geometric version of the Nadel product theorem], Bull. Soc. Math. France 119 (1991),
no. 4, 479--493 (French).


\bibitem{}\label{cheltsov-shramov}
I. Cheltsov, C. Shramov; 
\emph{Log canonical thresholds of smooth Fano threefolds.} 
Russian Mathematical Surveys \textbf{63} (2008), 859-958. 
Appendix by J.-P. Demailly; \emph{On Tian's invariant and log canonical thresholds}.

\bibitem{}\label{toric-cox-etal} 
D. A. Cox, J. B. Little, H. K. Schenck;  {\emph{Toric varieties.}}  
Graduate Studies in Mathematics, volume 214, American Mathematical Society (2011).


\bibitem{}\label{HMX2}
C.~D.~Hacon, J.~M$^{\rm c}$Kernan and C.~Xu;
\emph{ACC for log canonical thresholds}, Ann. of Math. (2) \textbf{180} (2014), no. 2, 523-571.

\bibitem{}\label{HMX}
C.~D.~Hacon, J.~M$^{\rm c}$Kernan and C.~Xu;
\emph{On the birational automorphisms of varieties of general type},
Ann. \ of Math. \ (2) \textbf{177} (2013), no. 3, 1077-1111.

\bibitem{}\label{HX}
C.~D.~Hacon and C.~Xu; 
\emph{Boundedness of log Calabi-Yau pairs of Fano type.}
 Math. Res. Lett. 22 (2015), no. 6, 1699-1716.

\bibitem{}\label{Isk-Prokh}
V. A. Iskovskikh and Yu. G. Prokhorov, \emph{Fano varieties}. Algebraic geometry. V., 
Encyclopaedia Math. Sci., vol. 47, Springer, Berlin, 1999.

\bibitem{}\label{Jiang-3} 
C. Jiang, \emph{Boundedness of $\Q$-Fano varieties with degrees and alpha-invariants bounded from below}, 
to appear in Annales scientifiques de l'ENS, arXiv:1705.02740.

\bibitem{}\label{Jiang-2} 
C. Jiang, \emph{Boundedness of anti-canonical volumes of singular log Fano threefolds}, 
to appear in Comm. Anal. Geom., arXiv:1411.6728v2. 

\bibitem{}\label{Jiang} 
C. Jiang, \emph{On birational boundedness of Fano fibrations}, 
Amer. J. Math. 140 (2018), no. 5, 1253-1276. 


\bibitem{}\label{kawakita} 
M. Kawakita; \textit{Inversion of adjunction on log canonicity}, 
Invent. Math. \textbf{167} (2007), 129-133.

\bibitem{}\label{kawamata-bnd-ext-ray}
Y.~Kawamata; \emph{On the length of an extremal rational curve},
Invent. \ Math. \ \textbf{105} (1991), no. 3, 609-611.

\bibitem{}\label{kawamata-term-3-folds} 
{Y. Kawamata; {\emph{Boundedness of $\bold Q$-Fano threefolds.}}  Proceedings of the International Conference on Algebra, Part 3 (Novosibirsk, 1989),  439--445, Contemp. Math., \textbf{131}, Part 3, Amer. Math. Soc., Providence,  RI, 1992.}

\bibitem{}\label{Kedlaya} 
K. S. Kedlaya; {\emph{More \'etale covers of affine spaces in positive characteristic.}} 
 Journal of Algebraic Geometry \textbf{14} (2005), 187-192.

\bibitem{}\label{Kollar-toroidal}
J.~Koll\'ar; \emph{Partial resolution by toroidal blow-ups.}
 Tunis. J. Math. 1 (2019), no. 1, 3-12.

\bibitem{}\label{Kollar-flip-abundance}
J.~Koll\'ar \'et al.; \emph{Flips and abundance for algebraic threefolds,}
Ast\'erisque No. \textbf{211} (1992).

\bibitem{}\label{KMM-smooth-fano} 
{J. Koll\'ar; Y. Miyaoka; S. Mori; {\emph{Rationally connectedness and boundedness of Fano manifolds.}} 
J. Di. Geom. \textbf{36} (1992), 765-769. }

\bibitem{}\label{KMMT-can-3-folds} 
{J. Koll\'ar; Y. Miyaoka; S. Mori; H. Takagi; {\emph{Boundedness of canonical $\bold Q$-Fano 3-folds.}} 
Proc. Japan Acad. Ser. A Math. Sci.  \textbf{76}  (2000),  no. 5, 73--77.}


\bibitem{}\label{kollar-mori}
J.~Koll\'ar and S.~Mori,
Birational geometry of algebraic varieties,
Cambridge Tracts in Math. \textbf{134},
Cambridge Univ.\ Press, 1998.


\bibitem{}\label{Lazar}
R. Lazarsfeld; \emph{Positivity in algebraic geometry II.} 
Springer (2004).


\bibitem{}\label{Lin}
J. Lin; \emph{Birational unboundedness of $Q$-Fano threefolds.} 
Int. Math. Res. Not. \textbf{6} (2003), 301-312.

\bibitem{}\label{MP}  {J. M$^{\rm c}$Kernan, Yu. Prokhorov; {\emph{Threefold thresholds.}} Manuscripta Math.
114  (2004),  no. 3, 281--304.  }

\bibitem{}\label{Nadel}
A. M. Nadel, \emph{The boundedness of degree of Fano varieties with Picard number one,} 
J. Amer. Math. Soc. 4 (1991), no. 4, 681--692.

\bibitem{}\label{Nikulin-3}
V. V. Nikulin, \emph{del Pezzo surfaces with log-terminal singularities. III.} (Russian) Izv. Akad. Nauk SSSR Ser. Mat. 53 (1989), no. 6, 1316--1334, 1338; translation in Math. USSR-Izv. 35 (1990), no. 3, 657--675. 

\bibitem{}\label{Nikulin-2}
V. V. Nikulin, \emph{del Pezzo surfaces with log-terminal singularities. II.} 
(Russian) Izv. Akad. Nauk SSSR Ser. Mat. 52 (1988), no. 5, 1032--1050, 1119; translation in Math. USSR-Izv. 33 (1989), no. 2, 355--372.

\bibitem{}\label{Nikulin}
V. V. Nikulin, \emph{Del Pezzo surfaces with log-terminal singularities,} 
Mat. Sb. 180 (1989), no. 2, 226--243, translation in Math. USSR-Sb. 66 (1990), no. 1, 231--248.


\bibitem{}\label{Prokhorov-plt-blowups}
{Yu. Prokhorov; {\emph{Blow-ups of canonical singularities.}} 
Algebra (Moscow, 1998), 301–317, de Gruyter, Berlin, 2000.}


\bibitem{}\label{PSh-II} 
{Yu. Prokhorov, V.V. Shokurov; {\emph{Towards the second main theorem on complements.}} 
J.  Algebraic Geometry, \textbf{18} (2009) 151-199.}

\bibitem{}\label{PSh-I} {Yu. Prokhorov; V.V. Shokurov; {\emph{The first fundamental Theorem on complements: from global to local.}} (Russian)  Izv. Ross. Akad.  Nauk Ser. Mat.  \textbf{65}  (2001),  no. 6, 99--128;  translation in  Izv. Math.  65  (2001),  no. 6, 1169--1196.}


\bibitem{}\label{Prokhorov-Shramov}
Yu.  Prokhorov  and  C.  Shramov; \emph{Jordan  property  for  Cremona  groups.}
Amer. J. Math.,\textbf{138} (2016), 403-418.


\bibitem{}\label{Prokhorov-Shramov-2}
Yu.  Prokhorov  and  C.  Shramov; \emph{Jordan  property  for groups of birational selfmaps.}
 Compositio Math.,  Volume 150, Issue 12 (2014), 2054-2072.


\bibitem{}\label{Serre} 
J.-P. Serre;\emph{ A Minkowski-style bound for the orders of the finite
 subgroups of the Cremona group of rank $2$ over an arbitrary field.}
Mosc. Math. J., \textbf{9} (2009):193-208.

\bibitem{}\label{Tian} 
G. Tian; \emph{On a set of polarized K\"ahler metrics on algebraic manifolds.}
J. Differential Geom. \textbf{32} (1990), 99-130.

\bibitem{}\label{Viehweg}
E. Viehweg; \emph{Quasi-Projective Moduli of Polarized Manifolds.}
Springer-Verlag, Berlin, 1995

\bibitem{}\label{Xu-plt}
C. Xu; \emph{Finiteness of algebraic fundamental groups.} Compositio Math., 150 (3) (2014), 409-414.

\bibitem{}\label{Zarhin}
Yu. G. Zarhin. \emph{Theta groups and products of abelian and rational varieties.} 
Proc. Edinburgh Math. Soc., 57 (2014):299-304.



\end{thebibliography}
\end{document}